\documentclass[a4paper,11pt]{article}

\usepackage[a4paper,margin=2.5cm]{geometry}

\usepackage[T1]{fontenc}
\usepackage{lmodern}    

\usepackage{setspace}
\usepackage{amsmath,amssymb,amsthm}

\usepackage{graphicx}

\usepackage[colorlinks=true,
            linkcolor=blue,
            citecolor=blue,
            urlcolor=blue]{hyperref}
\usepackage{colonequals}
\usepackage{comment}
\usepackage{cite}
\usepackage{hyperref}
\usepackage{url}
\usepackage{doi}
\usepackage{scalerel}[2014/03/10]
\usepackage[usestackEOL]{stackengine}
\def\intavg{\,\ThisStyle{\ensurestackMath{%
    \stackinset{c}{0\LMpt}{c}{0\LMpt}{\SavedStyle-}{\SavedStyle\phantom{\int}}}%
    \setbox0=\hbox{$\SavedStyle\int\,$}\kern-\wd0}\int}
\usepackage{amsthm}
\theoremstyle{plain}
\newtheorem{theorem}{Theorem}[section]

\theoremstyle{definition}
\newtheorem{definition}[theorem]{Definition}

\theoremstyle{lemma}
\newtheorem{lemma}[theorem]{Lemma}

\theoremstyle{Remark}
\newtheorem{remark}[theorem]{Remark}

\theoremstyle{proposition}

\theoremstyle{corollary}
\newtheorem{corollary}[theorem]{Corollary}

\theoremstyle{example}

\theoremstyle{assumption}

\usepackage{systeme}
\usepackage{authblk}

\title{Well-posedness for  the spatially inhomogeneous Fermi-Dirac-Fokker-Planck equation }
\author[]{Gayrat Toshpulatov\thanks{Corresponding author: {\tt gayrat.toshpulatov@uni-muenster.de}}
}

  \affil[]{\begin{small} Institut f\"ur Analysis und Numerik,\\ Fachbereich Mathematik und 
Informatik der Universit\"at M\"unster,\\ Orl\'eans-Ring 10, 48149 M\"unster, Germany.  \end{small}} 
\date{\today}
\begin{document}
\maketitle
\pagestyle{plain}
\begin{abstract} 
This paper is concerned with the spatially inhomogeneous Fermi-Dirac-Fokker-Planck equation posed on a classical kinetic phase space. The nonlinear factor $f(1-f)$ appears both in the transport term and in the collision part of the Fokker-Planck operator, accounting for the  exclusion principle of fermions. 

We prove the existence and uniqueness of global solutions without  any smallness or close-to-equilibrium assumptions. To construct solutions, we use a vanishing viscosity approximation by adding a Laplacian in the spatial variable. We derive uniform estimates for the regularized problem and obtain compactness in the space of functions of bounded variation  $(BV),$ allowing us to pass to the limit as the spatial viscosity vanishes. Since the limiting equation is degenerate with a nonlinear transport, we formulate an appropriate notion of  entropy solutions incorporating the kinetic structure of the equation. We prove uniqueness within this class and show that every vanishing-viscosity limit satisfies the entropy formulation. Our result provides a global well-posedness theory  beyond the perturbative regime.
\end{abstract}

\textcolor{black}{\begin{small}\textbf{Keywords:} Degenerate parabolic equation, Fermi-Dirac-Fokker-Planck equation, semiclassical kinetic equation,   entropy solutions, functions of bounded variation, vanishing viscosity approach.
\end{small}}
\tableofcontents

\section{Introduction}
One approach to studying the statistical evolution of a rarefied gas of particles is to describe the dynamics in terms of a kinetic equation for the particle distribution function $f$ in phase space. Such equations typically consist of a transport part, describing the free motion of the particles, and a collision part, accounting for their  interactions. With the advent of quantum mechanics, classical kinetic equations were extended to incorporate quantum statistical effects, in particular the Pauli exclusion-inclusion principle. This line of research was initiated by Nordheim \cite{Nordh} and by Uehling and Uhlenbeck \cite{U.U}. Since then, several kinetic models for interacting quantum particles subject to these principles have been introduced and studied; see, e.g., \cite{Book, Book1, Dolb, E.M1, E.M.Vel, Lu.Wenn}. Such models  have found applications in, among others, the theory of electron conductivity \cite{Nordh} and the modeling of electron interactions in semiconductor devices \cite{Book1}.

A characteristic feature of quantum kinetic equations is the presence of the factor $1\pm f$ in the collision operator. This factor accounts for the quantum statistical effects of indistinguishable particles and leads, at equilibrium, to the Bose-Einstein and Fermi-Dirac distributions, respectively.
Balazs \cite{Der} and Kaniadakis \cite{K} (see also \cite{KQ1, KQ2}) introduced a nonlinear kinetic model 
in which the factor $1\pm f$ appears not only in the Fokker–Planck collision operator but also in the transport part
\begin{equation}\label{DEq0}
\begin{cases}
\partial_t f+p\cdot \nabla_x (f+\kappa f^2)
=\operatorname{div}_{p}\big(\nabla_p f+p(f+\kappa f^2)\big), \, \, \, \, \, \,x,p\in \mathbb{R}^d,\,t>0,\\
f(0,x,p)=f_0(x,p).
\end{cases}
\end{equation}
Here, $t\geq0$, $x\in\mathbb{R}^d$, and $p\in\mathbb{R}^d$ denote time, position, and momentum, respectively. The unknown $f=f(t,x,p)\geq0$ represents the distribution function of indistinguishable quantum particles obeying the Pauli inclusion-exclusion principle.
 The parameter $\kappa$ distinguishes the fermionic and bosonic cases: $\kappa=-1$ corresponds to the Pauli exclusion principle for fermions, while $\kappa=1$ corresponds to the  inclusion principle for bosons. 
For $\kappa=0$, Equation \eqref{DEq0} reduces to the classical kinetic Fokker--Planck equation \cite{H.Risken}.

In the fermionic case, the distribution function is required to satisfy the constraint
$$
0\leq f\leq 1.
$$
For both values of $\kappa,$ 
whenever $f(t,x,p)$ is a sufficiently smooth solution of \eqref{DEq0}, one has  \textit{conservation of mass:}
\begin{equation*}
\int_{\mathbb{R}^{2d}}f(t,x,p)dpdx=\int_{\mathbb{R}^{2d}}f_0(x,p)dpdx, \, \, \, \, \forall\, t\geq 0.
\end{equation*}

If we simplify \eqref{DEq0} by neglecting the dependence of the solution on the $x$-variable, then we get the spatially homogeneous equation
\begin{equation}
\label{Eqhom}
\begin{cases}
\partial_t f(t,p)=\operatorname{div}_{p}(\nabla_pf(t,p)+p(f(t,p)+\kappa f^2(t,p)))\\
f(0,p)=f_0(p).
\end{cases}
\end{equation} 
The spatially homogeneous equation \eqref{Eqhom} has been studied comprehensively. In the fermionic case $\kappa=-1$, global existence and uniqueness of solutions in arbitrary dimension $d\in\mathbb{N}$ were established in \cite{Car} by using the fundamental solution of the linear Fokker--Planck operator together with suitable moment estimates. The same work also investigated the long-time behavior and showed that solutions converge exponentially fast to the corresponding Fermi-Dirac equilibrium as $t\to\infty$. The bosonic case $\kappa=1$ exhibits a substantially different behavior. Global existence and uniqueness were established in \cite{2D} in dimension $d=2$. In arbitrary dimension, local-in-time existence was obtained in \cite{Car}. However, as shown by Toscani \cite{To}, finite-time blow-up can occur in dimensions $d\geq3$. More precisely,  Toscani \cite{To} proved that in dimension
$d\geq 3$ there are certain solutions which blow up if the initial mass is sufficiently large or, for any given
supercritical mass, if the initial second moment is small enough. The continuation beyond this blow-up and long time behavior was studied in \cite{Num, Hopf}.    
 
 The spatially inhomogeneous equation \eqref{DEq0} is considerably more challenging. Existing results have mainly been obtained under suitable smallness or close-to-equilibrium assumptions. In the case of a periodic spatial domain, i.e., $x\in\mathbb{T}^d$, the existence of smooth solutions and their exponential convergence to the global equilibrium were established in \cite{NS} for initial data sufficiently close to equilibrium in an appropriate Sobolev space. The proof relies on an $H^1$ hypocoercivity method developed in the previous works \cite{NM, V}. For the whole-space problem, i.e., $x\in\mathbb{R}^3$, Luo and Zhang \cite{Chi} obtained  well-posedness and  algebraic decay rates of  solutions toward the global equilibrium 
  under a suitable  smallness assumption on the initial data in a Sobolev space. They used  the high-order Sobolev energy method developed by Guo \cite{Guo}.
We also mention the recent work \cite{APT}, where, under the assumption that a global weak solution exists, exponential convergence to the global equilibrium was established in a weighted $L^2$ space without assuming that the initial data is close to equilibrium.

  A natural question is whether, in the fermionic case,  one can prove existence and uniqueness of solutions without any smallness or close-to-equilibrium assumptions.
A major difficulty difficulty stems from the degenerate parabolic structure of  \eqref{DEq0}: the Fokker-Planck operator acts only in the momentum variable $p$, while there are no second-order derivatives in the spatial variable $x$. The transport operator, on the other hand, acts in the $x$ variable. Thus, the parabolic  regularization with respect to the  $x$-variable  is not available. 
A further difficulty arises from the nonlinear term $f(1-f)$, which appears in both the transport and the Fokker-Planck parts of the equation. This nonlinear structure prevents a direct application of standard existence and regularity theories for linear kinetic equations. In particular, because of this nonlinearity it is not clear if one can expect  hypoelliptic regularization (which is well-known  for the linear kinetic Fokker-Planck equation \cite{AT1,AT2, V, Hera, Zha}) to recover  regularity with respect to the $x$-variable.   Our analysis therefore requires new tools  and modifications of existing techniques.

 In this paper, we study the nonlinear, spatially inhomogeneous kinetic equation \eqref{DEq0} in the fermionic case $\kappa=-1$.   We prove the existence and uniqueness of global solutions without imposing any smallness or close-to-equilibrium assumptions on the initial data. The proof is based on the vanishing viscosity method: For a small parameter $\varepsilon>0$ we consider the following uniformly  parabolic equation
\begin{equation}\label{Eq}
\begin{cases}
\partial_t f_{\varepsilon}+p\cdot \nabla_x (f_{\varepsilon}-f_{\varepsilon}^2)=\mathrm{div}_{p}(\nabla_p f_{\varepsilon}+p(f_{\varepsilon}-f_{\varepsilon}^2))+\varepsilon \Delta_x f_{\varepsilon},\, \, \, \, \, \,x,p\in \mathbb{R}^d,\,t>0,\\
f_{\varepsilon}(0,x,p)=f_0(x,p).
\end{cases}
\end{equation}
This equation differs from \eqref{DEq0} with $\kappa=-1$ by the addition of the term $\varepsilon\Delta_x f_\varepsilon$, which makes it uniformly parabolic.
 We will prove existence and uniqueness of global solutions to this equation and obtain regularity estimates. Our analysis on \eqref{Eq} is motivated by the work \cite{Car} where existence and uniqueness of solutions were proven for the spatially homogeneous equation \eqref{Eqhom} in the fermionic case $\kappa=-1$. We extend the existence analysis in \cite{Car} to the spatially inhomogeneous  equation \eqref{Eq}.  Under appropriate regularity assumptions on the initial data, we then derive estimates that are uniform with respect to $\varepsilon$ and allow us to pass to the limit as $\varepsilon\to0$. In particular, up to the extraction of a subsequence, $f_{\varepsilon}$ converges to a $BV$ function $f$, and we show that this limit is a solution to   \eqref{DEq0} with $\kappa=-1$.

Since \eqref{DEq0} is degenerate and contains a nonlinear transport term, uniqueness cannot in general be expected at the level of weak solutions. Vol’pert and Hudjaev \cite{VH, Khu} introduced a notion of entropy solutions in  $BV$  for degenerate parabolic equations with bounded coefficients, and established existence and uniqueness of such solutions in one dimension under some regularity assumptions
on the initial data. Related works can be found in \cite{Ent1, Ent2, Ent3}. 
Motivated by  this work,   we introduce a class of \emph{entropy solutions} (see Definition \ref{def:gen.sol} below) adapted to the kinetic structure of \eqref{DEq0}.    Entropy solutions  are required to satisfy a suitable integral inequality which implies that such solutions satisfy  \eqref{DEq0} with $\kappa=-1$ in the sense of distributions and  equality of measures. We prove that this class enjoys a uniqueness property and that every limit obtained from the vanishing-viscosity approximation belongs to this class. Consequently, we obtain the existence and  uniqueness of entropy solutions  to \eqref{DEq0} with $\kappa=-1$.

 The organization of the paper is as follows. In Section 2, we introduce the notion of entropy solutions and state the main results. In Section 3, we show that entropy solutions are unique by proving that  $L^1-$difference  of two solutions is contracting as time progresses. Section 4 is devoted to the analysis of the uniformly parabolic equation \eqref{Eq}.We prove existence and uniqueness, regularity, and  some properties of  solutions. In Section 5, we obtain estimates for the solutions of \eqref{Eq} that are uniform with respect to \(\varepsilon\). We then use these estimates to pass to the limit as \(\varepsilon\to0\) and prove the main existence result.

\section{Main results}
We rewrite \eqref{DEq0} for the fermionic case $\kappa=-1$
\begin{equation}\label{DEq}
\begin{cases}
\partial_t f+p\cdot \nabla_x (f- f^2)
=\operatorname{div}_{p}\big(\nabla_p f+p(f- f^2)\big), \, \, \, \, \, \,x,p\in \mathbb{R}^d,\,t>0,\\
f(0,x,p)=f_0(x,p).
\end{cases}
\end{equation}
 We introduce the notion of entropy solutions:
\begin{definition}\label{def:gen.sol} \textit{Let $f_0\in L^1(\mathbb{R}^{2d})$ and  $0\leq f_0(x,p)\leq 1$ for almost every $(x,p)\in \mathbb{R}^{2d}$. 
A measurable  function $f: (0, \infty)\times \mathbb{R}^{2d}\to \mathbb{R}$   is called an \emph{entropy solution} of  \eqref{DEq}   with the initial data $f_0$  if: 
\begin{itemize}
\item $0\leq f(t,x,p)\leq 1$ for almost every $(t,x,p)\in (0,\infty)\times \mathbb{R}^{2d};$
\item $f\in L^1((0,T)\times \mathbb{R}^{2d})$, $f\in BV((0,T)\times \mathbb{R}^{2d})$,  and $\nabla_p f\in [L^2((0,T)\times \mathbb{R}^{2d})]^d$ for all $T>0$;
\item For any smooth compactly supported $\phi\in C_c^{\infty}((0, \infty)\times \mathbb{R}^{2d}) $ with $\phi\geq 0$ and any constant $a\in \mathbb{R}$
\begin{align}\label{def}
\int_0^{\infty}\int_{\mathbb{R}^{2d}}\mathrm{sign}(f-a)\big\{(f-a)\partial_t \phi &+(f-f^2-(a-a^2))(p\cdot \nabla_x \phi-p\cdot \nabla_p \phi)\nonumber \\
&-\nabla_p f\cdot\nabla_p \phi+d(a-a^2)\phi \big\}dxdpdt\geq 0;
\end{align}
\item There is a set $\mathcal{E}  $  of zero Lebesgue  measure on $(0,1)$ such that 
$$ \lim_{t \to 0,\, t\in (0,1) \setminus \mathcal{E}} \int_{\mathbb{R}^{2d}}|f(t,x,p)-f_0(x,p)|dxdp=0.$$ 
\end{itemize}
}\end{definition}
 Equation \eqref{def} holds for all $a\in \mathbb{R}$. The condition $0\leq f\leq 1$ provides that  the sign of $f-a$ is constant for $a<0$ and $a>1$. This fact and  the  nonnegativity of the left hand side of \eqref{def}  imply 
\begin{equation*}
\int_0^{\infty}\int_{\mathbb{R}^{2d}} \big\{f\partial_t \phi +(f-f^2)(p\cdot \nabla_x \phi-p\cdot \nabla_p \phi) -\nabla_p f\cdot\nabla_p \phi \big\}dxdpdt=0
\end{equation*}
for any smooth compactly supported $\phi\in C_c^{\infty}((0, \infty)\times \mathbb{R}^{2d}). $ Hence, an entropy solution satisfies \eqref{DEq} in the sense of distributions. Moreover, it follows from this equation that $\Delta_p f$ is a locally finite Radon measure, since  $\partial_t f$, $\partial_{x_i}f$, and $\partial_{p_i} f$ are locally finite Radon measures for all $i\in \{1,...,d\}$. Consequently, all terms in \eqref{DEq} define locally finite Radon measures, and \eqref{DEq} holds as an equality of locally finite Radon measures.

It is not difficult to check that every sufficiently smooth solution of \eqref{DEq} satisfying $0\leq f\leq 1$ also satisfies the entropy inequality \eqref{def}.

Our first main result is the $L^1$-contraction and uniqueness of entropy solutions.
\begin{theorem}[\bf{$L^1$-contraction and uniqueness}]\label{th:unique}
Let $d\in \mathbb{N}.$ Let $f$ and $g$ be two entropy solutions to \eqref{DEq} with the respective initial data $f_0,g_0\in L^1(\mathbb{R}^{2d})$ satisfying  $0\leq f_0(x,p),g_0(x,p)\leq 1$ for almost every $(x,p)\in \mathbb{R}^{2d}$. Then \begin{align*}
\int_{\mathbb{R}^{2d}}|f(t,x,p)-g(t,x,p)| dxdp\leq  \int_{\mathbb{R}^{2d}}|f_0(x,p)-g_0(x,p)| dxdp
\end{align*} 
for almost every $t\in (0,\infty).$
In particular, for any initial data $f_0$ as above, there is at most one entropy solution $f$. 
\end{theorem}
We define a nonlinear differential operator 
$$\mathcal{L}[f]\colonequals -p\cdot \nabla_x (f- f^2)+ \operatorname{div}_{p}\big(\nabla_p f+p(f- f^2)\big)$$
so that \eqref{DEq} can be written as $\partial_tf=\mathcal{L}[f].$
For $k\geq 1$ we also define a  space
$$L_{k}^1(\mathbb{R}^{2d}) \colonequals \Big\{g\in L^1(\mathbb{R}^{2d}):\, \int_{\mathbb{R}^{2d}}|p|^{k}|g(x,p)|dxdp<\infty \Big\}$$
and its norm
$$||g||_{L_{k}^1(\mathbb{R}^{2d})}\colonequals \int_{\mathbb{R}^{2d}}(1+|p|^{k})|g(x,p)|dxdp.$$
For $C\in \mathbb{R}$ we use  $\lceil  C \rceil$ to denote the smallest integer  larger or equal to $C$.

Our second main result is the global existence of entropy solutions to \eqref{DEq}.
\begin{theorem}[\bf{Global existence}]\label{th:global existence}
Let $d\in \mathbb{N},$  $k>2d$,  $f_0\in L^1_{k}(\mathbb{R}^{2d})\cap W^{1,1}(\mathbb{R}^{2d}) $,   $0\leq f_0(x,p)\leq 1$ for almost every $(x,p)\in \mathbb{R}^{2d},$ $\mathcal{L}[f_0]\in L^1(\mathbb{R}^{2d})$, and $\Delta_x f_0\in L^1(\mathbb{R}^{2d})$. Then there is a unique entropy solution $f: (0, \infty)\times \mathbb{R}^{2d}\to \mathbb{R}$   as in Definition \ref{def:gen.sol}. We have  
 $$  \int_{\mathbb{R}^{2d}}f(t,x,p)dpdx=\int_{\mathbb{R}^{2d}}f_0(x,p)dpdx$$
 for almost every $t\in (0,\infty)$.     For all $ \gamma\in [1, k]$, the $\gamma$-moment of $f$ with respect to $p$, i.e.,  $\int_{\mathbb{R}^{2d}}|p|^{\gamma}f(t,x,p)dxdp,$ is bounded by a positive polynomial $P_{\lceil \gamma/2 \rceil}(t)$   (whose coefficients depend only on  $\int_{\mathbb{R}^{2d}}f_0(x,p)dxdp,$ $\int_{\mathbb{R}^{2d}}|p|^{\gamma}f_0(x,p)dxdp,$ $\gamma$, and $d$) of degree $\lceil \gamma/2 \rceil$ for almost every $t\in (0,\infty)$. In particular, $f(t,\cdot,\cdot)\in L_k^1(\mathbb{R}^{2d})$ for almost every $t\in (0,\infty)$.
\end{theorem}
\begin{remark}
$(i)$ The condition $f_0\in L^1_{k}(\mathbb{R}^{2d})$ with $k>2d$ is needed to prove existence of global solutions; see Lemma \ref{lem:p-moment} and Theorem \ref{th:global} below.  We note that $ L_{k}^1(\mathbb{R}^{2d})\subset L_{\gamma}^1(\mathbb{R}^{2d}) $  for all $ \gamma \in [1, k].$\\
$(ii)$ The regularity assumptions $f_0\in  W^{1,1}(\mathbb{R}^{2d}) $, $\mathcal{L}[f_0]\in L^1(\mathbb{R}^{2d})$, and  $\Delta_x f_0\in L^1(\mathbb{R}^{2d})$ are needed to obtain   uniform bounds with respect to $\varepsilon\in (0,1]$ in $W^{1,1}((0, T)\times \mathbb{R}^{2d})$  (with $T>0$) for the solution of \eqref{Eq}, which imply compactness in $BV((0, T)\times \mathbb{R}^{2d});$ see Lemma \ref{l:epsilon} and the proof of Theorem \ref{th:global existence}. In view of $0\leq f_0\leq 1$, the condition $\mathcal{L}[f_0]\in L^1(\mathbb{R}^{2d})$ is satisfied if  $\Delta_p f_0$, $p\cdot \nabla_x f_0$, and $p\cdot \nabla_p f_0$ belong to $ L^1(\mathbb{R}^{2d})$.
\\
$(iii)$ The uniqueness of the solution in Theorem \ref{th:global existence}  follows directly from Theorem \ref{th:unique}.
\end{remark}
\section{Uniqueness of entropy solutions}
In this section we  show the uniqueness of entropy solutions, i.e., we prove Theorem \ref{th:unique}.   
\subsection{Some results concerning $BV$ functions}
Let $f$ be a $BV$ function in $(0,\infty)\times \mathbb{R}^{2d}$.
We say that $f$ has a \textit{Lebesgue point} at $(t,x,p)\in (0,\infty)\times \mathbb{R}^{2d}$
if there is $\tilde{f}\in \mathbb{R}$ such that 
 \begin{align*}
 \lim_{\rho \to 0}\intavg_{B_{\rho}(t,x,p)}|f(s,y,q)-\tilde{f}|dydqds=0,
 \end{align*}
where 
 $
 B_{\rho}(t,x,p)\colonequals \big\{(s,y,q)\in (0,\infty)\times \mathbb{R}^{2d}:\,|t-s|^2+|x-y|^2+|p-q|^2<\rho^2 \big\}.
$  
 We denote by $S_f \subset (0,\infty)\times \mathbb{R}^{2d}$ the set where this condition fails and it is called the \emph{approximate discontinuity set}; see \cite[Definition 3.63]{Amb}. We note that $S_f$ is a Borel set of zero Lebesgue measure and $\tilde{f}:(0,\infty)\times \mathbb{R}^{2d}\setminus S_f\to \mathbb{R}$ is Borel measurable coinciding with $f$ almost everywhere in $(0,\infty)\times \mathbb{R}^{2d}\setminus S_f$; see \cite[Proposition 3.64]{Amb}. 
 
 We say $(t,x,p)\in (0,\infty)\times \mathbb{R}^{2d}$ is an \emph{approximate jump point of $f$}      (see \cite[Definition 3.67]{Amb}) if there exists $f^+, f^{-}\in \mathbb{R},$ and a unit vector $\nu_f=(\nu_t,
 \nu_x,\nu_p)\in \mathbb{S}^{2d}$ (with $\nu_t\in \mathbb{R},$ $\nu_x \in \mathbb{R}^d$, $\nu_p\in \mathbb{R}^d$)  such that $f^{+}\neq f^{-}$ and 
 \begin{align*}
 \lim_{\rho \to 0}\intavg_{B_{\rho}^{+}(t,x,p;\nu)}|f(s,y,q)-f^{+}|dydqds=0,\,\,\,\,\,\lim_{\rho \to 0}\intavg_{B_{\rho}^{-}(t,x,p;\nu)}|f(s,y,q)-f^{-}|dydqds=0,
 \end{align*}
 where 
 \begin{align*}
 B_{\rho}^{+}(t,x,p;\nu)\colonequals \big\{(s,y,q)\in (0,\infty)\times \mathbb{R}^{2d}:&\,|t-s|^2+|x-y|^2+|p-q|^2<\rho^2,\\ & (s-t)\nu_t+(y-x)\cdot \nu_x+(q-p)\cdot \nu_p>0 \big\}
 \end{align*}
 and 
 \begin{align*}
 B_{\rho}^{-}(t,x,p;\nu)\colonequals \big\{(s,y,q)\in (0,\infty)\times \mathbb{R}^{2d}: &\,|t-s|^2+|x-y|^2+|p-q|^2<\rho^2,\\
 & (s-t)\nu_t+(y-x)\cdot \nu_x+(q-p)\cdot \nu_p<0 \big\}.
  \end{align*}
The set of approximate jump points is denoted by $J_f\subset (0,\infty)\times \mathbb{R}^{2d}$. It is obvious that $J_f\subset S_f$. 
 The triple $$(f^+(t,x,p), f^-(t,x,p), \nu_f(t,x,p))$$ is uniquely defined up to a permutations of $(f^+(t,x,p), f^-(t,x,p))$ and a change of sign of $\nu_f(t,x,p)$.  The set $J_f$ is a Borel set of zero Lebesgue measure and   
$$(f^+(t,x,p), f^-(t,x,p), \nu_f(t,x,p)): J_f\to \mathbb{R}\times \mathbb{R}\times \mathbb{S}^{2d}$$
 is a Borel map; see \cite[Proposition 3.69]{Amb}. 
 It is known that $\mathcal{H}^{2d}(S_f\setminus J_f)=0$ (see \cite[Theorem 3.78]{Amb}),  where $\mathcal{H}^{2d}$ is the Hausdorff measure. 
One can extend $f^{+}$ and $f^{-}$ to $(0,\infty)\times \mathbb{R}^{2d}\setminus S_f $ by setting $$f^+(t,x,p)=f^{-}(t,x,p)\colonequals \tilde{f}(t,x,p)\,\,\,\,\text{for}\,\,\,\,\, (t,x,p)\in (0,\infty)\times \mathbb{R}^{2d}\setminus S_f .$$ 
With this extension, the function $$\bar{f}(t,x,p)\colonequals \frac{f^+(t,x,p)+f^-(t,x,p)}{2} $$ is defined for $\mathcal{H}^{2d}$-almost everywhere  on $(0,\infty)\times \mathbb{R}^{2d}$. This function is called the \emph{mean value} of $f$.

\subsection{Jump points and $L^1$-contraction}
Let $f$ be an entropy solution as in Definition \ref{def:gen.sol}.
Since $\nabla_p f$ is in $ [L^2((0,T)\times \mathbb{R}^{2d})]^d$ for all $T>0$ and the Lebesgue measure of $J_f$ is zero, $\int_E\nabla_p f(t,x,p)dxdpdt$  is equal to zero for each Borel set $E\subset J_f\cap (0,T)$. By \cite[Theorem 3.78]{Amb} we have 
$$0=\int_E\nabla_p f(t,x,p)dxdpdt=-\int_E(f^+(t,x,p)-f^-(t,x,p))\nu_p(t,x,p)d\mathcal{H}^{2d},$$
where  $\nu_p(t,x,p)$ is the $p$-component of the normal $\nu_f$. Since $E\subset J_f\cap (0,T)$ and  $T>0$ are arbitrary and $f^{+}\neq f^{-}$, we get
\begin{align}\label{nu_p}
\nu_p(t,x,p)=0,\,\,\,\,\, \mathcal{H}^{2d}\text{-almost everywhere  on}\,\,\,J_f.
\end{align} 
This equation, \eqref{def}, and \cite[Theorem 1.3]{Khu} imply that  for any constant $a\in \mathbb{R}$
\begin{align*}
\Big[(\mathrm{sign}(&f^{+}(t,x,p)-a)-\mathrm{sign}(f^{-}(t,x,p)-a)\Big]\Big[\Big(\frac{f^{+}(t,x,p)+f^{-}(t,x,p)}{2}-a\Big)\nu_t(t,x,p)\\
&+\Big(\frac{f^{+}(t,x,p)-(f^+(t,x,p))^2+f^{-}(t,x,p)-(f^{-}(t,x,p))^2}{2}-(a-a^2)\Big)p\cdot \nu_x(t,x,p) \Big]\leq 0
\end{align*}
holds for $\mathcal{H}^{2d}$-almost everywhere  on $J_f$.  
The bound $0\leq f \leq 1$ implies $0\leq f^+,f^-\leq 1 .$ Hence, for  $a<0$ and $a>1$, the signs of $f^{+}(t,x,p)-a$ and $f^{-}(t,x,p)-a$ are constant. Using this fact and the above inequality, we get
\begin{align}\label{+=-}
(f^{+}(t,x,p)&-f^{-}(t,x,p))\nu_t(t,x,p)\nonumber\\
&+(f^{+}(t,x,p)-f^{-}(t,x,p))(1-f^{+}(t,x,p)-f^{-}(t,x,p)) p\cdot \nu_x(t,x,p)=0
\end{align}
for $\mathcal{H}^{2d}$-almost everywhere  on $J_f$. 
 Combining the above two equations, we get  
\begin{align}\label{+<-}
|f^{+}&(t,x,p)-a|\nu_t(t,x,p)+|f^{+}(t,x,p)-a|(1-f^{+}(t,x,p)-a)p\cdot \nu_x(t,x,p)\nonumber\\
&\leq |f^{-}(t,x,p)-a|\nu_t(t,x,p)+|f^{-}(t,x,p)-a|(1-f^{-}(t,x,p)-a)p\cdot \nu_x(t,x,p)
\end{align}
for $\mathcal{H}^{2d}$-almost everywhere  on $J_f$.

\begin{lemma} Let $f$ and $g$ be two entropy solutions to \eqref{DEq} as in Definition \ref{def:gen.sol}. Then 
\begin{align}\label{I>0}
 \int_0^{\infty}\int_{\mathbb{R}^{2d}}\mathrm{sign}(f-g)\big\{(f-g)\partial_t \phi &+(f-f^2-(g-g^2))(p\cdot \nabla_x \phi-p\cdot \nabla_p \phi)\nonumber \\
&-\nabla_p (f-g)\cdot\nabla_p \phi \big\}dxdpdt\geq 0
\end{align}
for all $\phi \in  C_c^{\infty}((0, \infty)\times \mathbb{R}^{2d})$ with $\phi\geq 0.$
\end{lemma}
\begin{proof} We denote the left hand side of \eqref{I>0} by $I(f,g,\phi).$ 
 Since \eqref{DEq} holds in the sense of equality of measures for $f$ and $g$, we have 
\begin{align}\label{new I}
I(f,g,\phi)=\int_0^{\infty}\int_{\mathbb{R}^{2d}}\mathrm{sign}(z)\big\{\partial_t (\phi z) &+p\cdot \nabla_x(\phi z(1-f-g))\nonumber\\ &- \mathrm{div}_p(p\phi z(1-f-g))
-\mathrm{div}_p(\phi \nabla_p z) \big\}dxdpdt,
\end{align}
where  $z\colonequals f-g$.
We use \eqref{nu_p} and \cite[Theorem 1.1]{Khu} to integrate by parts 
\begin{align*}
&I_1(f,g,\phi)\colonequals\int_0^{\infty}\int_{\mathbb{R}^{2d}}\mathrm{sign}(z)\big\{\partial_t (\phi z) +p\cdot \nabla_x(\phi z(1-f-g))- \mathrm{div}_p(p\phi z(1-f-g))\big\}dxdpdt\nonumber \\
&=-\int_{J_{z}}\phi(\mathrm{sign}(z^+)-\mathrm{sign}(z^-))\Big(\frac{z^++z^-}{2}\nu_t+\frac{ z^+(1-f^+-g^+)+ z^-(1-f^--g^-)}{2}p\cdot \nu_x\Big)d\mathcal{H}^{2d}.
\end{align*}
By writing and combining \eqref{+=-} for $f$ and $g$, we get 
\begin{align*}
z^+\nu_t+z^+(1-f^+-g^+)p\cdot \nu_x& =z^-\nu_t+z^-(1-f^--g^-)p\cdot \nu_x\\&=\frac{z^++z^-}{2}\nu_t+\frac{ z^+(1-f^+-g^+)+ z^-(1-f^--g^-)}{2}p\cdot \nu_x.
\end{align*}
We choose $a=g$ in \eqref{+<-} and get 
\begin{align*}
\mathrm{sign}(z^+)(z^+\nu_t+z^+(1-f^+-g^+)p\cdot \nu_x)\leq \mathrm{sign}(z^-)(z^-\nu_t+z^-(1-f^--g^-)p\cdot \nu_x).
\end{align*}
The last two equations imply
\begin{align*}
&(\mathrm{sign}(z^+)-\mathrm{sign}(z^-))\Big(\frac{z^++z^-}{2}\nu_t+\frac{ z^+(1-f^+-g^+)+ z^-(1-f^--g^-)}{2}p\cdot \nu_x\Big)\\
&=\mathrm{sign}(z^+)(z^+\nu_t+z^+(1-f^+-g^+)p\cdot \nu_x)-\mathrm{sign}(z^-)(z^-\nu_t+z^-(1-f^--g^-)p\cdot \nu_x\leq 0
\end{align*}
and so $I_1(f,g,\phi)\geq 0. $ 

The last term of  \eqref{new I} can be written as
$$I_2(f,g,\phi)\colonequals -\int_0^{\infty}\int_{\mathbb{R}^{2d}}\mathrm{sign}(z)  \mathrm{div}_p(\phi \nabla_p z) dxdpdt=-\lim_{\tau\to 0}\int_0^{\infty}\int_{\mathbb{R}^{2d}}\mathrm{sign}_{\tau}(z)  \mathrm{div}_p(\phi \nabla_p z) dxdpdt,$$ 
where $\mathrm{sign}_{\tau}(z) $ is the function defined in \eqref{sign} below. Integration by parts and $\mathrm{sign}_{\tau}'(z)\geq 0$ provide $$-\int_{\mathbb{R}^{2d}}\mathrm{sign}_{\tau}(z)  \mathrm{div}_p(\phi \nabla_p z) dxdpdt=\int_{\mathbb{R}^{2d}}\mathrm{sign}_{\tau}'(z)  \phi |\nabla_p z|^2 dxdpdt\geq 0$$
and so  $I_2(f,g,\phi)\geq 0.$

\end{proof}

\begin{proof}[\textbf{Proof of Theorem \ref{th:unique}}]
Equation \eqref{I>0} can written as 
\begin{align*}
\int_0^{\infty}\int_{\mathbb{R}^{2d}}\big\{|f-g|\partial_t \phi &+|f-g|(1-f-g)(p\cdot \nabla_x \phi-p\cdot \nabla_p \phi)\nonumber \\
&-\nabla_p (|f-g|)\cdot\nabla_p \phi \big\}dxdpdt\geq 0.
\end{align*}
This means that the distribution
$$-\partial_t(|f-g|)+(p\cdot \nabla_p-p\cdot\nabla_x)(|f-g|(1-f-g))+\Delta_p(|f-g|)$$
is a nonnegative  measure.  The first two terms are locally finite Radon measures. Hence, the last term is also a locally finite Radon measure.  Let $\phi_R(x,p)$ be a smooth nonnegative function such that $\phi_R(x,p)=1$ if $\sqrt{|x|^2+|p|^2}\leq R$ and $\phi_R(x,p)=0$ if $\sqrt{|x|^2+|p|^2}\geq R+1$.   We use the nonnegativity of this measure, multiply it  by $\phi_R$, and integrate 
over $[0,T]\times \mathbb{R}^{2d}$
\begin{align*}
\int_{\mathbb{R}^{2d}}|f(T)-g(T)|\phi_R dxdp\leq & \int_{\mathbb{R}^{2d}}|f_0-g_0|\phi_R dxdp\\
&+\int_0^{T}\int_{\mathbb{R}^{2d}}|f-g|\Big[(1-f-g)(p\cdot \nabla_x \phi-p\cdot \nabla_p \phi)+\Delta_p \phi\Big]dxdpdt \\
\leq & \int_{\mathbb{R}^{2d}}|f_0-g_0|\phi_R dxdp+C\int_0^{T}\int_{\sqrt{|x|^2+|p|^2}\geq R}|f-g|dxdpdt. 
\end{align*}
The terms $p\cdot \nabla_x \phi-p\cdot \nabla_p \phi$ and $\Delta_p \phi$ are bounded uniformly with respect to $R$ (see \eqref{zeta}-\eqref{gr.z} below for a similar function).
Using  $f,g\in L^1([0,T]\times \mathbb{R}^{2d})$ and  the monotone convergence theorem, we let $R\to \infty$ and obtain 
\begin{align*}
\int_{\mathbb{R}^{2d}}|f(T)-g(T)| dxdp\leq  \int_{\mathbb{R}^{2d}}|f_0-g_0| dxdp. 
\end{align*} 
\end{proof}

\section{Solution to the uniformly  parabolic equation}
In this section we analyze the uniformly parabolic equation \eqref{Eq}: we establish   existence and uniqueness of global solutions and their regularity.

To prove existence and uniqueness of solutions to \eqref{Eq}, we consider the fundamental solution of its linear part:
Let us define
\begin{equation*}
G(t,x,p, x_0,p_0)=\frac{1}{(4\pi^2 \gamma(t))^{d/2}}\mathrm{exp}\left(-\frac{\Big|\sqrt{\frac{\alpha(t)}{\gamma(t)}}X-\frac{\beta(t)}{\sqrt{\alpha(t)\gamma(t)}}P\Big|^2+\Big|\frac{P}{\sqrt{\alpha(t)}}\Big|^2}{2}\right),
\end{equation*}
where $$X\colonequals x-x_0-(1-e^{-t})p_0,\,P\colonequals p-e^{-t}p_0,$$
$$\alpha(t)\colonequals  1-e^{-2t},\, \beta(t)\colonequals  (1-e^{-t})^2,$$ 
$$\gamma(t)\colonequals  2(1-e^{-t})[(1+\varepsilon)t(1+e^{-t})-2(1-e^{-t})].$$
It can be proven (see \cite{Hor}, \cite[Lemma 2.5]{AE}, \cite{Fried}) that $\int_{\mathbb{R}^{2d}}G(t,x,p,x_0,p_0)f_0(x_0,p_0)dx_0dp_0$ is a unique smooth solution of the linear parabolic equation  
$$\begin{cases}
\partial_t f_{\varepsilon}+p\cdot \nabla_x f_{\varepsilon}=\mathrm{div}_{p}(\nabla_p f_{\varepsilon}+pf_{\varepsilon})+\varepsilon \Delta_x f_{\varepsilon},\, \, \, \, \, \,x,p\in \mathbb{R}^d,\,t>0,\\
f_{\varepsilon}(0,x,p)=f_0(x,p).
\end{cases}$$
Since this equation is the linear part of \eqref{Eq}, the Duhamel principle suggests to convert the nonlinear equation \eqref{Eq}   to a nonlinear 
integral equation
\begin{align}\label{Duh}
&f_{\varepsilon}(t,x,p)=\int_{\mathbb{R}^{2d}}G(t,x,p,x_0,p_0)f_0(x_0,p_0)dx_0dp_0\nonumber\\
&+\int_0^t\int_{\mathbb{R}^{2d}}f_{\varepsilon}^2(s,x_0,p_0)p_0\cdot(\nabla_{p_0}G(t-s,x,p,x_0,p_0)-\nabla_{x_0}G(t-s,x,p,x_0,p_0))dx_0dp_0ds.
\end{align}
We mention that a function $f$ satisfying \eqref{Duh} is called a \emph{mild solution} to \eqref{Eq}, see
\cite[Section 6.1]{Pazy}. Using the properties of $G(t,x,p,x_0,p_0)$ and fixed point arguments we show
that there is a unique solution to this integral equation. 
\subsection{Auxiliary lemmas}
In this subsection we prove some estimates involving the terms on the right hand side of \eqref{Duh}. These estimates  will be used later.

Let $k\geq 1$. We define the spaces $\Upsilon \colonequals L^{\infty}(\mathbb{R}^{2d})\cap L^{1}_1(\mathbb{R}^{2d})\cap L^{k}_1(\mathbb{R}^{2d})$ and $\Upsilon_T\colonequals C([0,T]; \Upsilon)$ for any $T>0$ with the norms
\begin{equation}\label{norm}
||f(t)||_{\Upsilon}\colonequals \mathrm{max}\big\{||f(t)||_{L^{\infty}(\mathbb{R}^{2d})}, ||f(t)||_{L^{1}_1(\mathbb{R}^{2d})}, ||f(t)||_{L^{k}_1(\mathbb{R}^{2d})}\big\}\,\,\,\,\,\text{and}\,\,\,\,\,||f||_{\Upsilon_T}\colonequals \mathrm{max}_{0\leq t\leq T}||f(t)||_{\Upsilon},
\end{equation}
 where 
$||f||_{L^k_1(\mathbb{R}^{2d})}\colonequals \left(\int_{\mathbb{R}^{2d}}(1+|p|)^k|f|^kdxdp\right)^{\frac{1}{k}}$.

We define a linear operator
\begin{align*}
\mathcal{G}[u](t,x,p)\colonequals \int_{\mathbb{R}^{2d}}G(t,x,p,x_0,p_0)u(x_0,p_0)dx_0dp_0
\end{align*}
for functions $u\in L^k_1$ and $t>0.$

The following lemma will be used to prove the local existence result  in Theorem \ref{th:loc}. 
\begin{lemma} Let $k\geq 1$.  There is a constant $C_1=C_1(k,d)>0$ such that, for all $u\in  L^k_1$, 
\begin{align}\label{fun1}
||\mathcal{G}[u](t)||_{L^k_1}\leq C_1e^{dt}||u||_{L^k_1},\,\,\,\forall\, t> 0.
\end{align}
\end{lemma}
\begin{proof} 
Using $1+|p|\leq 1+|P|+|e^{-t}p_0|\leq (1+|P|)(1+|e^{-t}p_0|)$ we estimate 
\begin{align}\label{*}
(1+|p|)|\mathcal{G}[u](t,x,p)|\leq \int_{\mathbb{R}^{2d}}&\frac{(1+|P|)}{(4\pi^2 \gamma(t))^{d/2}}\mathrm{exp}\left(-\frac{\Big|\sqrt{\frac{\alpha(t)}{\gamma(t)}}X-\frac{\beta(t)}{\sqrt{\alpha(t)\gamma(t)}}P\Big|^2+\Big|\frac{P}{\sqrt{\alpha(t)}}\Big|^2}{2}\right)\times  \nonumber \\
& \times (1+|e^{-t}p_0|)|u(x_0,p_0)|dx_0dp_0.
\end{align}
Let us denote
$$\xi\colonequals x-(e^t-1)p,\,\eta\colonequals e^t p.$$
Then we have $$X=\xi-x_0+(1-e^{-t})(\eta-p_0),\,P=e^{-t}(\eta-p_0).$$
Hence,  we can consider the right hand side of \eqref{*} 
as a convolution with respect to the variables $\xi$ and $\eta$, i.e., we can write it as $\int_{\mathbb{R}^{2d}} U(t,\xi-x_0, \eta-p_0)V(t, x_0,p_0)dx_0dp_0$ with $$U(t,\xi-x_0, \eta-p_0)\colonequals \frac{(1+|P|)}{(4\pi^2 \gamma(t))^{d/2}}\mathrm{exp}\left(-\frac{\Big|\sqrt{\frac{\alpha(t)}{\gamma(t)}}X-\frac{\beta(t)}{\sqrt{\alpha(t)\gamma(t)}}P\Big|^2+\Big|\frac{P}{\sqrt{\alpha(t)}}\Big|^2}{2}\right)$$ and  $$V(t, x_0,p_0)\colonequals (1+|e^{-t}p_0|)|u(x_0,p_0)|.$$ This allows us to use the Young inequality
\begin{align*}
||\mathcal{G}[u](t)||_{L^k_1}\leq C& \left(\int_{\mathbb{R}^{2d}}(1+|e^{-t}p|)^k|u(x,p)|^kdxdp\right)^{\frac{1}{k}}
\times \\
&\times \int_{\mathbb{R}^{2d}}\frac{(1+|\tilde{P}|)}{(4\pi^2 \gamma(t))^{d/2}}\mathrm{exp}\left(-\frac{\Big|\sqrt{\frac{\alpha(t)}{\gamma(t)}}\tilde{X}-\frac{\beta(t)}{\sqrt{\alpha(t)\gamma(t)}}\tilde{P}\Big|^2+\Big|\frac{\tilde{P}}{\sqrt{\alpha(t)}}\Big|^2}{2}\right)d\xi d\eta,
\end{align*}
where  $\tilde{X}\colonequals \xi+(1-e^{-t})\eta,$  $\tilde{P}\colonequals e^{-t}\eta,$ and  $C$ is a positive constant. By changing the variables $\xi=\tilde{X}-(e^t-1)\tilde{P}$ and $\eta=e^t \tilde{P},$ we compute
\begin{align*}
&\int_{\mathbb{R}^{2d}}\frac{(1+|\tilde{P}|)}{(4\pi^2 \gamma(t))^{d/2}}\mathrm{exp}\left(-\frac{\Big|\sqrt{\frac{\alpha(t)}{\gamma(t)}}\tilde{X}-\frac{\beta(t)}{\sqrt{\alpha(t)\gamma(t)}}\tilde{P}\Big|^2+\Big|\frac{\tilde{P}}{\sqrt{\alpha(t)}}\Big|^2}{2}\right)d\xi d\eta\\
&= e^{dt}\int_{\mathbb{R}^{2d}}\frac{(1+|\tilde{P}|)}{(4\pi^2 \gamma(t))^{d/2}}\mathrm{exp}\left(-\frac{\Big|\sqrt{\frac{\alpha(t)}{\gamma(t)}}\tilde{X}-\frac{\beta(t)}{\sqrt{\alpha(t)\gamma(t)}}\tilde{P}\Big|^2+\Big|\frac{\tilde{P}}{\sqrt{\alpha(t)}}\Big|^2}{2}\right)d\tilde{X} d\tilde{P}.
\end{align*}
We change the variables once more $Z\colonequals \sqrt{\frac{\alpha(t)}{\gamma(t)}}\tilde{X}-\frac{\beta(t)}{\sqrt{\alpha(t)\gamma(t)}}\tilde{P}$ and $Y\colonequals \frac{\tilde{P}}{\sqrt{\alpha(t)}}$ by noting $d\tilde{X}d\tilde{P}=(\gamma(t))^{d/2} dZdY$
\begin{align*}
& e^{dt}\int_{\mathbb{R}^{2d}}\frac{(1+|\tilde{P}|)}{(4\pi^2 \gamma(t))^{d/2}}\mathrm{exp}\left(-\frac{\Big|\sqrt{\frac{\alpha(t)}{\gamma(t)}}\tilde{X}-\frac{\beta(t)}{\sqrt{\alpha(t)\gamma(t)}}\tilde{P}\Big|^2+\Big|\frac{\tilde{P}}{\sqrt{\alpha(t)}}\Big|^2}{2}\right)d\tilde{X} d\tilde{P}\\
&=e^{dt}\int_{\mathbb{R}^{2d}}\frac{(1+|\sqrt{\alpha(t)}{Y}|)}{(4\pi^2)^{d/2}}\mathrm{exp}\left(-\frac{|Z|^2+|Y|^2}{2}\right)dZ dY
.
\end{align*}
The above estimates show
\begin{align*}
||\mathcal{G}[u](t)||_{L^k_1}\leq  Ce^{dt}&\left(\int_{\mathbb{R}^{2d}}(1+|e^{-t}p|)^k|u(x,p)|^kdxdp\right)^{\frac{1}{k}}\times\\
&\times \int_{\mathbb{R}^{2d}}\frac{(1+|\sqrt{\alpha(t)}{Y}|)}{(4\pi^2)^{d/2}}\mathrm{exp}\left(-\frac{|Z|^2+|Y|^2}{2}\right)dZdY.
\end{align*}
Since $1+|e^{-t}p|\leq 1+|p|,$ $1+|\sqrt{\alpha(t)}{Y}|\leq 1+|{Y}|,$ and   $\int_{\mathbb{R}^{2d}}\frac{(1+|\sqrt{\alpha(t)}{Y}|)}{(4\pi^2)^{d/2}}\mathrm{exp}\left(-\frac{|Z|^2+|Y|^2}{2}\right)dZdY$ is finite, we obtain the claimed estimate.  
\end{proof}

For all $t> 0$ and for  vector functions $w:\mathbb{R}^d\times \mathbb{R}^d\to \mathbb{R}^d$,  we define
\begin{equation*}
\mathcal{F}[w](t,x,p)\colonequals \int_{\mathbb{R}^{2d}}w(x_0,p_0)\cdot(\nabla_{p_0}G(t,x,p,x_0,p_0)-\nabla_{x_0}G(t,x,p,x_0,p_0))dx_0dp_0.
\end{equation*}
Let us denote
$$\xi\colonequals x-(e^t-1)p,\,\eta\colonequals e^t p.$$ Then we can consider $\mathcal{F}[w]$ as a convolution with respect to $\xi$ and $\eta$, i.e., we can write it as $\int_{\mathbb{R}^{2d}} U(t,\xi-x_0, \eta-p_0)\cdot w (x_0,p_0)dx_0dp_0$ with $$U(t,\xi-x_0, \eta-p_0)\colonequals \nabla_{p_0}G(t,x,p,x_0,p_0)-\nabla_{x_0}G(t,x,p,x_0,p_0).$$  Hence, we can write
\begin{equation*}
\mathcal{F}[w](t,x,p)= \int_{\mathbb{R}^{2d}}w(x-(e^t-1)p-x_0, e^tp-p_0)\cdot U(t,x_0, p_0) dx_0dp_0.
\end{equation*} 
This helps us compute
\begin{align}\label{x_i-der}
\partial_{x_i}\mathcal{F}[w](t,x,p)=\int_{\mathbb{R}^{2d}}\partial_{x_{0,i}}w(x_0,p_0)\cdot(\nabla_{p_0}G(t,x,p,x_0,p_0)-\nabla_{x_0}G(t,x,p,x_0,p_0))dx_0dp_0
\end{align}
 and 
\begin{align}\label{p_i-der}
\partial_{p_i}\mathcal{F}[w](t,x,p)
=\int_{\mathbb{R}^{2d}}[(1-e^t)\partial_{x_{0,i}}w(x_0,p_0)&+e^t\partial_{p_{0,i}}w(x_0,p_0)]\cdot(\nabla_{p_0}G(t,x,p,x_0,p_0)\nonumber \\
&-\nabla_{x_0}G(t,x,p,x_0,p_0))dx_0dp_0
\end{align}
for all for all $x_0=(x_{0,1},...,x_{0,d})\in \mathbb{R}^d$ and  $p_0=(p_{0,1},...,p_{0,d})\in \mathbb{R}^d.$ Equations \eqref{x_i-der} and \eqref{p_i-der} will be used to prove the regularity result in Theorem \ref{th:reg}.

The following lemma will be used to prove the local existence result  in Theorem \ref{th:loc}.
\begin{lemma}\label{lem2}
 Let $1\leq q\leq k\leq \infty$ and $T>0$.    There are constants $C_2=C_2(k,q,d, \varepsilon, T)>0$ and $C_3=C_3(k,q,d, \varepsilon, T)>0$ such that 
\begin{align}\label{fun.est11}
||\mathcal{F}[w](t)||_{L^k}\leq \frac{C_2e^{d(1+\frac{1}{k}-\frac{1}{q}-\frac{1}{d})t}}{ t^{d(\frac{1}{q}-\frac{1}{k})+\frac{1}{2}}}||w||_{L^q},\,\,\,\,\forall\,t\in (0,T]
\end{align}
and 
\begin{align}\label{fun.est}
||\mathcal{F}[w](t)||_{L^k_1}\leq \frac{C_3e^{d(1+\frac{1}{k}-\frac{1}{q}-\frac{1}{d})t}}{ t^{d(\frac{1}{q}-\frac{1}{k})+\frac{1}{2}}}||w||_{L^q_1},\,\,\,\,\forall\,t\in (0, T]
\end{align}
hold for all vector functions $w:\mathbb{R}^d\times \mathbb{R}^d\to \mathbb{R}^d$   with $|w|\in L^q_1$.
\end{lemma}
\begin{proof} We only prove \eqref{fun.est} and the proof of \eqref{fun.est11} follows similarly.

 We compute 
 \begin{align*}
 (\nabla_{p_0}-&\nabla_{x_0})G(t,x,p,x_0,p_0)\\
 &=-e^{-t}G(t,x,p,x_0,p_0)\left[\left(\sqrt{\frac{\alpha}{\gamma}}X-\frac{\beta}{\sqrt{\alpha \gamma}}P\right)\left(\sqrt{\frac{\alpha}{\gamma}}+\frac{\beta}{\sqrt{\alpha \gamma}}\right)-\frac{1}{\sqrt{\alpha}}\frac{P}{\sqrt{\alpha}}\right]
 \end{align*}
and so
 \begin{align*}
\mathcal{F}[w](t,x,p)\colonequals -e^{-t} \int_{\mathbb{R}^{2d}}G(t,x,p,x_0,p_0) w(x_0,p_0)\cdot\Big[\Big(\sqrt{\frac{\alpha}{\gamma}}X-\frac{\beta}{\sqrt{\alpha \gamma}}P\Big)&\Big(\sqrt{\frac{\alpha}{\gamma}}+\frac{\beta}{\sqrt{\alpha \gamma}}\Big)\\
&-\frac{1}{\sqrt{\alpha}}\frac{P}{\sqrt{\alpha}}\Big]dx_0dp_0. 
\end{align*}
Using $1+|p|\leq (1+|P|)(1+|e^{-t}p_0|)$ we estimate  
\begin{align}\label{**}
(1&+|p|)\mathcal{F}[w](t,x,p)\leq   e^{-t} \int_{\mathbb{R}^{2d}}(1+|e^{-t}p_0|)|w(x_0,p_0)|\times\nonumber\\
 &\times \left|\left(\sqrt{\frac{\alpha}{\gamma}}X-\frac{\beta}{\sqrt{\alpha \gamma}}P\right)\left(\sqrt{\frac{\alpha}{\gamma}}+\frac{\beta}{\sqrt{\alpha \gamma}}\right)-\frac{1}{\sqrt{\alpha}}\frac{P}{\sqrt{\alpha}}\right| (1+|P|)G(t,x,p,x_0,p_0)dx_0dp_0.
\end{align}
Let us denote
$$\xi\colonequals x-(e^t-1)p,\,\eta\colonequals e^t p.$$
Then we have $$X=\xi-x_0+(1-e^{-t})(\eta-p_0),\,P=e^{-t}(\eta-p_0).$$
Hence,  we can consider
the right hand side of \eqref{**}
as a convolution with respect to the variables $\xi$ and $\eta,$ i.e., we can write it as $\int_{\mathbb{R}^{2d}} U(t,\xi-x_0,\eta-p_0)V(t,x_0,p_0)dx_0dp_0$ with 
$$U(t,\xi-x_0,\eta-p_0)\colonequals \left|\left(\sqrt{\frac{\alpha}{\gamma}}X-\frac{\beta}{\sqrt{\alpha \gamma}}P\right)\left(\sqrt{\frac{\alpha}{\gamma}}+\frac{\beta}{\sqrt{\alpha \gamma}}\right)-\frac{1}{\sqrt{\alpha}}\frac{P}{\sqrt{\alpha}}\right| (1+|P|)G(t,x,p,x_0,p_0)$$
and $$V(t,x_0,p_0)\colonequals (1+|e^{-t}p_0|)|w(x_0,p_0)|.$$
 This allows us to use the Young inequality
\begin{align*}
&||\mathcal{F}[w](t)||_{L^k_1}\leq Ce^{-t} \left(\int_{\mathbb{R}^{2d}}(1+|e^{-t}p|)^q|w|^qdxdp\right)^{\frac{1}{q}}\times\\ 
&\times \left( \int_{\mathbb{R}^{2d}}\left|\left(\sqrt{\frac{\alpha}{\gamma}}\tilde{X}-\frac{\beta}{\sqrt{\alpha \gamma}}\tilde{P}\right)\left(\sqrt{\frac{\alpha}{\gamma}}+\frac{\beta}{\sqrt{\alpha \gamma}}\right)-\frac{1}{\sqrt{\alpha}}\frac{\tilde{P}}{\sqrt{\alpha}}\right|^r (1+|\tilde{P}|)^r \tilde{G}^r(t,\tilde{X},\tilde{P}) d\xi d\eta \right)^{\frac{1}{r}},
\end{align*}
where $1/k+1=1/q+1/r,$ $\tilde{X}\colonequals \xi+(1-e^{-t})\eta,$ $\tilde{P}\colonequals e^{-t}\eta,$ and $$\tilde{G}(t,\tilde{X},\tilde{P})\colonequals \frac{1}{(4\pi^2 \gamma(t))^{d/2}}\mathrm{exp}\left(-\frac{\Big|\sqrt{\frac{\alpha(t)}{\gamma(t)}}\tilde{X}-\frac{\beta(t)}{\sqrt{\alpha(t)\gamma(t)}}\tilde{P}\Big|^2+\Big|\frac{\tilde{P}}{\sqrt{\alpha(t)}}\Big|^2}{2}\right).$$ 
Here, we use $C$ to denote positive constants which can vary from one line to another.
By changing the variables $\xi=\tilde{X}-(e^t-1)\tilde{P}$ and $\eta=e^t \tilde{P},$ we compute
\begin{align*}
&\int_{\mathbb{R}^{2d}}\left|\left(\sqrt{\frac{\alpha}{\gamma}}\tilde{X}-\frac{\beta}{\sqrt{\alpha \gamma}}\tilde{P}\right)\left(\sqrt{\frac{\alpha}{\gamma}}+\frac{\beta}{\sqrt{\alpha \gamma}}\right)-\frac{1}{\sqrt{\alpha}}\frac{\tilde{P}}{\sqrt{\alpha}}\right|^r (1+|\tilde{P}|)^r \tilde{G}^r(t,\tilde{X},\tilde{P}) d\xi d\eta\\&=e^{dt}\int_{\mathbb{R}^{2d}}\left|\left(\sqrt{\frac{\alpha}{\gamma}}\tilde{X}-\frac{\beta}{\sqrt{\alpha \gamma}}\tilde{P}\right)\left(\sqrt{\frac{\alpha}{\gamma}}+\frac{\beta}{\sqrt{\alpha \gamma}}\right)-\frac{1}{\sqrt{\alpha}}\frac{\tilde{P}}{\sqrt{\alpha}}\right|^r (1+|\tilde{P}|)^r \tilde{G}^r(t,\tilde{X},\tilde{P}) d\tilde{X} d\tilde{P}.
\end{align*} 
We change the variables once more $Z\colonequals \sqrt{\frac{\alpha(t)}{\gamma(t)}}\tilde{X}-\frac{\beta(t)}{\sqrt{\alpha(t)\gamma(t)}}\tilde{P}$ and $Y\colonequals \frac{\tilde{P}}{\sqrt{\alpha(t)}}$
\begin{align*}
&e^{dt}\int_{\mathbb{R}^{2d}}\left|\left(\sqrt{\frac{\alpha}{\gamma}}\tilde{X}-\frac{\beta}{\sqrt{\alpha \gamma}}\tilde{P}\right)\left(\sqrt{\frac{\alpha}{\gamma}}+\frac{\beta}{\sqrt{\alpha \gamma}}\right)-\frac{1}{\sqrt{\alpha}}\frac{\tilde{P}}{\sqrt{\alpha}}\right|^r (1+|\tilde{P}|)^r \tilde{G}^r(t,\tilde{X},\tilde{P}) d\tilde{X} d\tilde{P}\\&=\frac{e^{dt}}{(4\pi^2)^{dr/2} (\gamma(t))^{\frac{d(r-1)}{2}}}\int_{\mathbb{R}^{2d}}\left|Z\left(\sqrt{\frac{\alpha}{\gamma}}+\frac{\beta}{\sqrt{\alpha\gamma}}\right)-\frac{1}{\sqrt{\alpha}}Y\right|^r (1+|\sqrt{\alpha}Y|)^r e^{-\frac{r(|Z|^2+|Y|^2)}{2}} dZ dY.
\end{align*} 
We use  $\sqrt{\alpha(t)}\leq 1$ and $\left|Z\left(\sqrt{\frac{\alpha}{\gamma}}+\frac{\beta}{\sqrt{\alpha\gamma}}\right)-\frac{1}{\sqrt{\alpha}}Y\right|^r \leq C\left(\left|\sqrt{\frac{\alpha}{\gamma}}+\frac{\beta}{\sqrt{\alpha\gamma}}\right|^r|Z|^r+\frac{1}{\alpha^{r/2}}|Y|^r \right)$
to estimate 
\begin{align*}
&\frac{e^{dt}}{(4\pi^2)^{dr/2} (\gamma(t))^{\frac{d(r-1)}{2}}}\int_{\mathbb{R}^{2d}}\left|Z\left(\sqrt{\frac{\alpha}{\gamma}}+\frac{\beta}{\sqrt{\alpha\gamma}}\right)-\frac{1}{\sqrt{\alpha}}Y\right|^r (1+|\sqrt{\alpha}Y|)^r e^{-\frac{r(|Z|^2+|Y|^2)}{2}} dZ dY\\&\leq \frac{Ce^{dt}}{(4\pi^2)^{dr/2} (\gamma(t))^{\frac{d(r-1)}{2}}} \left|\sqrt{\frac{\alpha}{\gamma}}+\frac{\beta}{\sqrt{\alpha\gamma}}\right|^r\int_{\mathbb{R}^{2d}}|Z|^r(1+|Y|)^r e^{-\frac{r(|Z|^2+|Y|^2)}{2}} dZ dY\\
&+\frac{Ce^{dt}}{(4\pi^2)^{dr/2} (\gamma(t))^{\frac{d(r-1)}{2}}\alpha^{r/2}}\int_{\mathbb{R}^{2d}}|Y|^r(1+|Y|)^r e^{-\frac{r(|Z|^2+|Y|^2)}{2}} dZ dY.
\end{align*} 
Since  $\frac{1}{(4\pi^2)^{dr/2}}\int_{\mathbb{R}^{2d}}|Z|^r(1+|Y|)^r e^{-\frac{r(|Z|^2+|Y|^2)}{2}} dZ dY$ and $\frac{1}{(4\pi^2)^{dr/2}}\int_{\mathbb{R}^{2d}}|Y|^r(1+|Y|)^r e^{-\frac{r(|Z|^2+|Y|^2)}{2}} dZ dY$ are finite, we obtain
\begin{align*}
&\int_{\mathbb{R}^{2d}}\left|\left(\sqrt{\frac{\alpha}{\gamma}}\tilde{X}-\frac{\beta}{\sqrt{\alpha \gamma}}\tilde{P}\right)\left(\sqrt{\frac{\alpha}{\gamma}}+\frac{\beta}{\sqrt{\alpha \gamma}}\right)-\frac{1}{\sqrt{\alpha}}\frac{\tilde{P}}{\sqrt{\alpha}}\right|^r (1+|\tilde{P}|)^r \tilde{G}^r(t,\tilde{X},\tilde{P}) d\xi d\eta\\
&
\leq \frac{Ce^{dt}}{ (\gamma(t))^{\frac{d(r-1)}{2}}}\left( \left|\sqrt{\frac{\alpha}{\gamma}}+\frac{\beta}{\sqrt{\alpha\gamma}}\right|^r+\alpha^{-r/2}\right).
\end{align*}
The above estimates let us conclude
\begin{align*}
&||\mathcal{F}[w](t)||_{L^k_1}\leq \frac{Ce^{(\frac{d}{r}-1)t}}{ (\gamma(t))^{\frac{d(1-1/r)}{2}}}\left( \left|\sqrt{\frac{\alpha}{\gamma}}+\frac{\beta}{\sqrt{\alpha\gamma}}\right|+\alpha^{-1/2}\right)\left(\int_{\mathbb{R}^{2d}}(1+|p|)^q|w|^qdxdp\right)^{\frac{1}{q}}.
\end{align*} 
We can estimate
$\frac{1}{\sqrt{\gamma(t)}}\leq \frac{C}{t}$ and $\left|\sqrt{\frac{\alpha}{\gamma}}+\frac{\beta}{\sqrt{\alpha\gamma}}\right|+\alpha^{-1/2}\leq \frac{C}{t^{\frac{1}{2}}}$ for all $t\in (0,T].$ These estimates and $1/r=1+1/k-1/q$ provide 
\begin{align*}
&||\mathcal{F}[w](t)||_{L^k_1}\leq \frac{Ce^{d(1+\frac{1}{k}-\frac{1}{q}-\frac{1}{d})t}}{ t^{d(\frac{1}{q}-\frac{1}{k})+\frac{1}{2}}}\left(\int_{\mathbb{R}^{2d}}(1+|p|)^q|w|^qdxdp\right)^{\frac{1}{q}}.
\end{align*} 
\end{proof} 

The following lemma will be used to prove the regularity result in Theorem \ref{th:reg}. 
\begin{lemma}  Let $1\leq q\leq k\leq \infty$ and $T>0$.   There are constants $C_4=C_4(k,q,d, T)>0$ and $C_5=C_5(k,q,d,T)>0$ such that 
\begin{equation}\label{x-der}
\left|\left| \nabla_x \int_{\mathbb{R}^{2d}} G(t,x,p,x_0,p_0)u(x_0,p_0)dx_0dp_0\right|\right|_{L^k_1}\leq \frac{C_4e^{d(1+\frac{1}{k}-\frac{1}{q}-\frac{1}{d})t}}{ t^{d(\frac{1}{q}-\frac{1}{k})+\frac{1}{2}}}||u||_{L^q_1}
\end{equation}
and
\begin{equation}\label{p-der}
\left|\left| \nabla_p \int_{\mathbb{R}^{2d}} G(t,x,p,x_0,p_0)u(x_0,p_0)dx_0dp_0\right|\right|_{L^k_1}\leq \frac{C_5e^{d(1+\frac{1}{k}-\frac{1}{q}-\frac{1}{d})t}}{ t^{d(\frac{1}{q}-\frac{1}{k})+\frac{1}{2}}}||u||_{L^q_1}
\end{equation}
hold for all functions $u\in L^q_1$ and $t\in (0,T]$.
\end{lemma}
\begin{proof} 
We can easily compute that 
\begin{align*}
&\nabla_x \int_{\mathbb{R}^{2d}} G(t,x,p,x_0,p_0)u(x_0,p_0)dx_0dp_0\\
&=-\int_{\mathbb{R}^{2d}}G(t,x,p,x_0,p_0)\left(\sqrt{\frac{\alpha}{\gamma}}X-\frac{\beta}{\sqrt{\alpha\gamma}}P\right)\sqrt{\frac{\alpha}{\gamma}} u(x_0,p_0)dx_0dp_0
\end{align*}
and 
\begin{align*}
&\nabla_p \int_{\mathbb{R}^{2d}} G(t,x,p,x_0,p_0)u(x_0,p_0)dx_0dp_0\\
&=\int_{\mathbb{R}^{2d}} \left[\left(\sqrt{\frac{\alpha}{\gamma}}X-\frac{\beta}{\sqrt{\alpha \gamma}}P\right)\frac{\beta}{\sqrt{\alpha \gamma}}-\frac{1}{\sqrt{\alpha}} \frac{P}{\sqrt{\alpha}}\right] G(t,x,p,x_0,p_0) u(x_0,p_0)dx_0dp_0.
\end{align*}
The arguments in Lemma \ref{lem2} yield
\begin{align*}
&\left|\left| \nabla_x \int_{\mathbb{R}^{2d}} G(t,x,p,x_0,p_0)u(x_0,p_0)dx_0dp_0\right|\right|_{L^k_1}\leq \frac{Ce^{(\frac{d}{r}-1)t}}{ (\gamma(t))^{\frac{d(1-1/r)}{2}}}\sqrt{\frac{\alpha(t)}{\gamma(t)}}||u||_{L^q_1}
\end{align*} 
and 
\begin{align*}
&\left|\left| \nabla_p \int_{\mathbb{R}^{2d}} G(t,x,p,x_0,p_0)u(x_0,p_0)dx_0dp_0\right|\right|_{L^k_1}\leq \frac{Ce^{(\frac{d}{r}-1)t}}{ (\gamma(t))^{\frac{d(1-1/r)}{2}}}\left|\frac{\beta (t)}{\sqrt{\alpha (t) \gamma (t)}}+\frac{1}{\sqrt{\alpha(t)}}\right| ||u||_{L^q_1},
\end{align*}
where $\frac{1}{r}=1+\frac{1}{k}-\frac{1}{q}.$ Here, we use $C$ to denote positive constants which can vary from one line to another. We can estimate
$\frac{1}{\sqrt{\gamma(t)}}\leq \frac{C}{t}$, $\sqrt{\frac{\alpha(t)}{\gamma(t)}}\leq \frac{C}{\sqrt{t}}$, and $\frac{\beta (t)}{\sqrt{\alpha(t) \gamma (t)}}+\frac{1}{\sqrt{\alpha(t)}}\leq \frac{C}{\sqrt{t}}$ for all $t\in (0,T].$ These estimates complete the proof. 
\end{proof}

\subsection{Local existence and uniqueness}
\begin{theorem}\label{th:loc} Let $k>2d$ and $\Upsilon \colonequals L^{\infty}(\mathbb{R}^{2d})\cap L^{1}_1(\mathbb{R}^{2d})\cap L^{k}_1(\mathbb{R}^{2d})$ with the norm in \eqref{norm}.   Then, for every  $f_0\in \Upsilon$, there exists $T_{\mathrm{max}}\in (0,\infty]$  such that \eqref{Eq} has a unique mild  solution $f_{\varepsilon}\in C([0,T_{\mathrm{max}}); \Upsilon)$ with $f_{\varepsilon}(0,x,p)=f_0(x,p).$ If $T_{\mathrm{max}}<\infty,$ then at least one of the limits 
$$\lim_{t\nearrow T_{\mathrm{max}}} ||f_{\varepsilon}(t)||_{L^{\infty}},\,\,\,\,\lim_{t\nearrow T_{\mathrm{max}}} ||f_{\varepsilon}(t)||_{L^{1}_1}, \,\,\,\text{and}\,\,\,\lim_{t\nearrow T_{\mathrm{max}}} ||f_{\varepsilon}(t)||_{L^{k}_1}$$
is infinite. 
\end{theorem}

\begin{proof} We omit the dependence of the solution $f\colonequals f_{\varepsilon}$ on $\varepsilon$ for notational convenience. We use $C$ to denote positive constants which can vary from one line to another.

   Let $T>0.$ In the following we assume $t\in [0, T].$
We consider the following operator
\begin{align*}
&\Gamma[f](t,x,p)=\int_{\mathbb{R}^{2d}}G(t,x,p,x_0,p_0)f_0(x_0,p_0)dx_0dp_0\nonumber\\
&+\int_0^t\int_{\mathbb{R}^{2d}}f^2(s,x_0,p_0)p_0\cdot(\nabla_{p_0}G(t-s,x,p,x_0,p_0)-\nabla_{x_0}G(t-s,x,p,x_0,p_0))dx_0dp_0ds.
\end{align*}
for functions $f\in \Upsilon_T$. We will  show that there is a fixed point of $\Gamma$ in $\Upsilon_T$.

\textbf{Step 1: Estimates on $\Gamma[f]$ in $\Upsilon_T$:} 
 One can check 
$$\int_{\mathbb{R}^{2d}}G(t,x,p,x_0,p_0)dx_0dp_0=e^{dt}$$
and so $\left|\left|\int_{\mathbb{R}^{2d}}G(t,\cdot,\cdot,x_0,p_0)f_0(x_0,p_0)dx_0dp_0\right|\right|_{L^{\infty}(\mathbb{R}^{2d})}\leq e^{dt}||f_0||_{L^{\infty}(\mathbb{R}^{2d})}.$  This equation and  \eqref{fun.est11} provide 
\begin{align*} 
||\Gamma[f](t)||_{L^{\infty}}&\leq e^{dt}||f_0||_{L^{\infty}}+\int_0^t \frac{Ce^{d(1-\frac{1}{k}-\frac{1}{d})(t-s)}}{ (t-s)^{\frac{d}{k}+\frac{1}{2}}}||\,|p|f^2(s)||_{L^k}ds.
\end{align*}
Using $|p|^k\leq (1+|p|)^k$ we estimate 
\begin{align*}
||\,|p|f^2(s)||_{L^k}&\leq ||f(s)||_{L^{\infty}}\left(\int_{\mathbb{R}^{2d}} |p|^k|f(s)|^kdxdp\right)^{\frac{1}{k}}\nonumber \\
&\leq ||f(s)||_{L^{\infty}}\left(\int_{\mathbb{R}^{2d}}(1+ |p|)^k|f(s)|^kdxdp\right)^{\frac{1}{k}}\nonumber \\
&=||f(s)||_{L^{\infty}}||f(s)||_{L^k_1}.
\end{align*}
 The above two estimates imply
\begin{align}\label{infty} 
||\Gamma[f](t)||_{L^{\infty}}&\leq e^{dt}||f_0||_{L^{\infty}}+\int_0^t \frac{Ce^{d(1-\frac{1}{k}-\frac{1}{d})(t-s)}}{ (t-s)^{\frac{d}{k}+\frac{1}{2}}}||f(s)||_{L^{\infty}}||f(s)||_{L^k_1} ds\nonumber\\
& \leq e^{dt}||f_0||_{L^{\infty}}+\int_0^t \frac{Ce^{d(1-\frac{1}{k}-\frac{1}{d})(t-s)}}{ (t-s)^{\frac{d}{k}+\frac{1}{2}}}||f(s)||^2_{\Upsilon}ds\nonumber \\
&\leq e^{dt}||f_0||_{L^{\infty}}+I_1(t) ||f||^2_{\Upsilon_T},
\end{align}
where $$I_1(t)\colonequals \int_0^t \frac{Ce^{d(1-\frac{1}{k}-\frac{1}{d})(t-s)}}{ (t-s)^{\frac{d}{k}+\frac{1}{2}}} ds$$
is finite because of $k>2d.$

We estimate the $L^k_1$-norm  using \eqref{fun1} and \eqref{fun.est} 
$$
||\Gamma[f](t)||_{L^{k}_1}\leq C e^{dt}||f_0||_{L^{k}_1}+\int_0^t \frac{Ce^{d(1-\frac{1}{k}-\frac{1}{d})(t-s)}}{ (t-s)^{\frac{d}{k}+\frac{1}{2}}}||\,|p|f^2(s)||_{L^{k/2}_1}ds.$$
 We estimate 
 \begin{align*}
 ||\,|p|f^2(s)||_{L^{k/2}_1}=\left(\int_{\mathbb{R}^{2d}} (1+|p|)^{k/2}|p|^{k/2}|f(s)|^kdxdp \right)^{\frac{2}{k}}\leq ||f(s)||_{L_1^k}^2.
 \end{align*}
 The above two estimates imply
\begin{align} \label{km}
||\Gamma[f](t)||_{L^{k}_1}&\leq C e^{dt}||f_0||_{L^{k}_1}+\int_0^t \frac{Ce^{d(1-\frac{1}{k}-\frac{1}{d})(t-s)}}{ (t-s)^{\frac{d}{k}+\frac{1}{2}}}||f(s)||_{L^k_1}^{2}ds \nonumber \\
&\leq  C e^{dt}||f_0||_{L^{k}_1}+\int_0^t \frac{Ce^{d(1-\frac{1}{k}-\frac{1}{d})(t-s)}}{ (t-s)^{\frac{d}{k}+\frac{1}{2}}} ||f(s)||_{\Upsilon}^2ds\nonumber \\
&\leq  C e^{dt}||f_0||_{L^{k}_1}+I_2(t) ||f||_{\Upsilon_T}^2,
\end{align}
where   $$I_2(t)\colonequals \int_0^t \frac{Ce^{d(1-\frac{1}{k}-\frac{1}{d})(t-s)}}{ (t-s)^{\frac{d}{k}+\frac{1}{2}}}ds$$
is finite because of $k>2d.$

We estimate the $L^1_1$-norm using  \eqref{fun.est}
\begin{align*} 
||\Gamma[f](t)||_{L^{1}_1}&\leq C e^{dt}||f_0||_{L^{1}_1}+\int_0^t \frac{Ce^{(d-1)(t-s)}}{ (t-s)^{\frac{1}{2}}}||\,|p|f^2(s)||_{L^1_1}ds.
\end{align*}
Using $k>2d\geq 2$ we estimate by interpolation
\begin{align*}
||\,|p|f^2(s)||_{L^1_1}&=\int_{\mathbb{R}^{2d}}(1+|p|)|p|f^2(s,x,p)dxdp
\\
&\leq \int_{\mathbb{R}^{2d}}(1+|p|)^2 f^2(s,x,p)dxdp\\
&\leq \left(\int_{\mathbb{R}^{2d}}(1+|p|)|f(s,x,p)|dxdp\right)^{\frac{k-2}{k-1}}\left(\int_{\mathbb{R}^{2d}}(1+|p|)^k|f(s,x,p)|^kdxdp\right)^{\frac{1}{k-1}}\\
&\leq C||f(s)||^{\frac{k-2}{k-1}}_{L^1_1}||f(s)||_{L^k_1}^{\frac{k}{k-1}}.
\end{align*}
Consequently, we get 
\begin{align}\label{11} 
||\Gamma[f](t)||_{L^{1}_1}&\leq C e^{dt}||f_0||_{L^{1}_1}+\int_0^t \frac{Ce^{(d-1)(t-s)}}{ (t-s)^{\frac{1}{2}}}||f(s)||^{\frac{k-2}{k-1}}_{L^1_1}||f(s)||_{L^k_1}^{\frac{k}{k-1}}ds\nonumber \\
&\leq C e^{dt}||f_0||_{L^{1}_1}+\int_0^t \frac{Ce^{(d-1)(t-s)}}{ (t-s)^{\frac{1}{2}}}||f(s)||_{\Upsilon}^2ds\nonumber \\
&\leq C e^{dt}||f_0||_{L^{1}_1}+I_3(t) ||f||_{\Upsilon_T}^2,
\end{align}
where $I_3(t)\colonequals \int_0^t \frac{Ce^{(d-1)(t-s)}}{ (t-s)^{\frac{1}{2}}}$
is finite.

The above estimates in \eqref{infty}, \eqref{km}, and \eqref{11} show that there are positive constants $C_1(T)\geq 1$ and $C_2(T)>0$ with $C_2(T)\to 0$ as $T\to 0$ such that 
$$||\Gamma[f]||_{\Upsilon_T}\leq C_1(T)||f_0||_{\Upsilon}+C_2(T)||f||^2_{\Upsilon_T}.$$
This shows that $\Gamma$ maps $\Upsilon_T$ to itself.

\textbf{Step 2: Fixed point in $\Upsilon_T$:}  Let $f,g\in \Upsilon_T$.
We compute
\begin{align*}
&||\Gamma[f]-\Gamma[g]||_{\Upsilon_T}\\&=\left|\left|\int_0^t\int_{\mathbb{R}^{2d}}(f(s)+g(s))(f(s)-g(s))p_0\cdot(\nabla_{p_0}G(t-s)-\nabla_{x_0}G(t-s))dx_0dp_0ds\right|\right|_{\Upsilon_T}\\
&\leq \sup_{t\in[0,T]}||f(t)+g(t)||_{L^{\infty}}\left|\left|\int_0^t\int_{\mathbb{R}^{2d}}(f(s)-g(s))p_0\cdot(\nabla_{p_0}G(t-s)-\nabla_{x_0}G(t-s))dx_0dp_0ds\right|\right|_{\Upsilon_T}.
\end{align*}
We can use \eqref{fun.est} and repeat the same arguments in  \eqref{infty}, \eqref{km}, and \eqref{11} to get
\begin{align*}
||\Gamma[f]-\Gamma[g]||_{\Upsilon_T}\leq C_2(T) (||f||_{\Upsilon_T}+||g||_{\Upsilon_T})\left|\left|f-g\right|\right|_{\Upsilon_T}. 
\end{align*}
We choose a small $T>0$ such that
 $$C_2(T)< \frac{1}{4C_1(T)||f_0||_{\Upsilon}}.$$   
 We denote $R\colonequals 2C_1(T)||f_0||_{\Upsilon}$ and $\delta\colonequals 2RC_2(T)\in (0,1)$.   We then have $$||\Gamma[f]||_{\Upsilon_T}<R\,\,\,\,\text{and}\,\,\,\,\,||\Gamma[f]-\Gamma[g]||_{\Upsilon_T}\leq \delta ||f-g||_{\Upsilon_T}$$ for all $f,g\in \Upsilon_T$ with $||f||_{\Upsilon_T}\leq R$ and  $||g||_{\Upsilon_T}\leq R$.  This provides  that there is a unique fixed point $f$ of $\Gamma$ in $\Upsilon_T$, and hence a unique solution to the integral equation in \eqref{Duh} for $t\in [0,T]$. 
 
\textbf{Step 3: Extension to $[0,T_{\text{max}})$:}  We note that $T$ depends on $|| f_0||_{\Upsilon}$. From what we have just proved it follows that if $f$ is a  solution on the interval $[0,T]$ it can be extended on the interval $[0, T+T_1]$ with some $T_1>0$ by defining on $[T, T+T_1],$ $f(t)=f_1(t)$ where $f_1(t)$ is the solution of \begin{align*}
&f_1(t,x,p)=\int_{\mathbb{R}^{2d}}G(t-T,x,p,x_0,p_0)f(T, x_0,p_0)dx_0dp_0\nonumber\\
&+\int_T^t\int_{\mathbb{R}^{2d}}f_1^2(s,x_0,p_0)p_0\cdot(\nabla_{p_0}G(t-s,x,p,x_0,p_0)-\nabla_{x_0}G(t-s,x,p,x_0,p_0))dx_0dp_0ds.
\end{align*}
Moreover, $T_1 $ depends on $|| f(T)||_{\Upsilon}.$ We can continue this process and extend the existence time interval.  
Let $[0, T_{\mathrm{max}})$ be the maximal interval of  existence of the solution. If $T_{\mathrm{max}}<\infty$ then at least one of the limits
$$\lim_{t\nearrow T_{\mathrm{max}}} ||f(t)||_{L^{\infty}},\,\,\,\,\lim_{t\nearrow T_{\mathrm{max}}} ||f(t)||_{L^{1}_1}, \,\,\,\text{and}\,\,\,\lim_{t\nearrow T_{\mathrm{max}}} ||f(t)||_{L^{k}_1}$$
is infinite. Otherwise there is a sequence $T_n\nearrow T_{\mathrm{max}},$ $n\in \mathbb{N},$ such that $||f(t_n)||_{ \Upsilon}$  is bounded. This would imply what we have just proved that for each $T_n,$ near enough to $T_{\mathrm{max}},$ the solution $f$ on the interval $[0,T_n]$ can be extended to the interval $[0, T_n+\theta],$ where $\theta >0$ is independent of $T_n$ and  hence $f$ can be extended beyond $T_{\mathrm{max}}. $ This contradicts the definition of $T_{\mathrm{max}}.$
\end{proof}

\subsection{Regularity}
\begin{theorem}\label{th:reg}
Let $k>2d$, $\Upsilon \colonequals L^{\infty}(\mathbb{R}^{2d})\cap L^{1}_1(\mathbb{R}^{2d})\cap L^{k}_1(\mathbb{R}^{2d})$ with the norm in \eqref{norm}, and  $f_0\in \Upsilon$. Let $f_{\varepsilon}\in C([0,T_{\mathrm{max}}); \Upsilon)$ with $f_{\varepsilon}(0,x,p)=f_0(x,p)$  be the solution  to \eqref{Eq}  in Theorem \ref{th:loc}.  Then $t^{\frac{1}{2}}|\nabla_x f_{\varepsilon}(t,x,p)| $ and $t^{\frac{1}{2}}|\nabla_p f_{\varepsilon}(t,x,p)| $ belong to $L^{\infty}((0,T); L^k_1\cap L^1_1(\mathbb{R}^{2d}))$ for all $T<T_{\mathrm{max}}$. Moreover, $f_{\varepsilon}$ is in $C^{\infty}((0,T_{\mathrm{max}})\times \mathbb{R}^{2d}).$
\end{theorem}
\begin{proof} We omit the dependence of the solution $f\colonequals f_{\varepsilon}$ on $\varepsilon$ for notational convenience.  We divide the proof into three parts. In the following $T>0$ denotes  any constant such that $T<T_{\mathrm{max}}.$  We use $C$ to denote positive constants which can vary from one line to another.
\\[2mm]
\textbf{Part 1: Estimates on $\nabla_x f$:} 
Using \eqref{x_i-der} we take  the gradient with respect to $x$ in the integral equation 
\begin{align}\label{x-derInt}
& \nabla_x f(t,x,p)=\nabla_x \int_{\mathbb{R}^{2d}}G(t,x,p,x_0,p_0)f_0(x_0,p_0)dx_0dp_0\nonumber\\
&+\int_0^t\int_{\mathbb{R}^{2d}}2f(s,x_0,p_0)\nabla_{x_0}f(s,x_0,p_0)p_0\cdot(\nabla_{p_0}-\nabla_{x_0})G(t-s,x,p,x_0,p_0)dx_0dp_0ds.
\end{align}
Let us denote 
$$\mathcal{X}(t)\colonequals \max\{||\nabla_x f(t)||_{L^k_1}, ||\nabla_x f(t)||_{L^1_1}\}$$
for $t>0$. Equation \eqref{x-der} implies
\begin{equation}\label{x1}
\left|\left| \nabla_x \int_{\mathbb{R}^{2d}} G(t,x,p,x_0,p_0)f_0(x_0,p_0)dx_0dp_0\right|\right|_{L^k_1}\leq \frac{C e^{(d-1)t}}{ t^{\frac{1}{2}}}||f||_{L^k_1}.
\end{equation} 
 We use \eqref{fun.est} (with $q=k/2$) to  estimate the second term on the right hand side of \eqref{x-derInt}
 \begin{align}\label{x11}
 \left|\left|\int_0^t\int_{\mathbb{R}^{2d}}2f(s,x_0,p_0)\nabla_{x_0}f(s,x_0,p_0)p_0\cdot(\nabla_{p_0}-\nabla_{x_0})G(t-s,x,p,x_0,p_0)dx_0dp_0ds \right|\right|_{L^k_1}\nonumber \\
 \leq \int_0^t\frac{Ce^{d(1-\frac{1}{k}-\frac{1}{d})(t-s)}}{ (t-s)^{\frac{d}{k}+\frac{1}{2}}}||\,|p|f(s)|\nabla_x f(s)|\,||_{L^{k/2}_1}.
\end{align} 
Using $(1+|p|)^{k/2}|p|^{k/2}\leq (1+|p|)^k$ and the H\"older inequality, we estimate 
\begin{align*}
||f(s)|p|\,|\nabla_x f(s)|\,||_{L^{k/2}_1}&=\left(\int_{\mathbb{R}^{2d}}(1+|p|)^{k/2}|p|^{k/2}|f(s,x,p)|^{k/2}|\nabla_xf(s,x,p)|^{k/2}dxdp\right)^{2/k}\\
&\leq \left(\int_{\mathbb{R}^{2d}}(1+|p|)^{k}|f(s,x,p)|^{k/2}|\nabla_xf(s,x,p)|^{k/2}dxdp\right)^{2/k}\\
&\leq  ||f(s)||_{L_1^k(\mathbb{R}^{2d})}||\nabla_x f(s)||_{L_1^k(\mathbb{R}^{2d})}.
\end{align*}
These estimates show 
 \begin{align}\label{x2}
 \left|\left|\int_0^t\int_{\mathbb{R}^{2d}}2f(s,x_0,p_0)\nabla_{x_0}f(s,x_0,p_0)p_0\cdot(\nabla_{p_0}-\nabla_{x_0})G(t-s,x,p,x_0,p_0)dx_0dp_0ds \right|\right|_{L^k_1}\nonumber \\
 \leq ||f||_{\Upsilon_T}\int_0^t\frac{Ce^{d(1-\frac{1}{k}-\frac{1}{d})(t-s)}}{ (t-s)^{\frac{d}{k}+\frac{1}{2}}}  ||\nabla_x f(s)||_{L^k_1}ds.
\end{align}
Equations \eqref{x1} and \eqref{x2} imply 
\begin{align}\label{x33}
 ||\nabla_x f(t)||_{L^k_1} &\leq \frac{Ce^{(d-1)t}}{ t^{\frac{1}{2}}}||f_0||_{L^k_1}+||f||_{\Upsilon_T}\int_0^t\frac{Ce^{d(1-\frac{1}{k}-\frac{1}{d})(t-s)}}{ (t-s)^{\frac{d}{k}+\frac{1}{2}}}  ||\nabla_x f(s)||_{L^k_1}ds\nonumber\\
 &\leq  \frac{Ce^{(d-1)t}}{ t^{\frac{1}{2}}}||f_0||_{L^k_1}+||f||_{\Upsilon_T}\int_0^t\frac{Ce^{d(1-\frac{1}{k}-\frac{1}{d})(T-s)}}{ (t-s)^{\frac{d}{k}+\frac{1}{2}}}  \mathcal{X}(s)ds
\end{align}

We now estimate the $L^1_1$-norm of $\nabla_x f$: Equation \eqref{x1}   holds with $k=1.$  We use \eqref{fun.est} (with $k=q=1$) to  estimate the second term on the right hand side of \eqref{x-derInt}
\begin{align*}
 \left|\left|\int_0^t\int_{\mathbb{R}^{2d}}2f(s,x_0,p_0)\nabla_{x_0}f(s,x_0,p_0)p_0\cdot(\nabla_{p_0}-\nabla_{x_0})G(t-s,x,p,x_0,p_0)dx_0dp_0ds \right|\right|_{L^1_1}\nonumber \\
 \leq \int_0^t\frac{Ce^{(d-1)(t-s)}}{ (t-s)^{\frac{1}{2}}}||\,|p|f(s)|\nabla_x f(s)|\,||_{L^{1}_1}.
\end{align*} 
This leads to estimate
 \begin{align*}
 ||\, |p|f(s)|\nabla_x f(s)|\,||_{L^1_1}&\leq C||f(s)||_{L^1_1}^{\frac{k-2}{2(k-1)}}||f(s)||_{L^k_1}^{\frac{k}{2(k-1)}}||\nabla_x f(s)||_{L^1_1}^{\frac{k-2}{2(k-1)}}||\nabla_xf(s)||_{L_1^k}^{\frac{k}{2(k-1)}}\\
 &\leq ||f||_{\Upsilon_T}\mathcal{X}(s),
 \end{align*}
where we used the (extended) H\"older inequality \cite[Section 4.2, Remark 2]{Bre}. The above estimate and \eqref{x11} provide
\begin{align}\label{xL^1_1}
 \left|\left|\int_0^t\int_{\mathbb{R}^{2d}}2f(s,x_0,p_0)\nabla_{x_0}f(s,x_0,p_0)p_0\cdot(\nabla_{p_0}-\nabla_{x_0})G(t-s,x,p,x_0,p_0)dx_0dp_0ds \right|\right|_{L^1_1}\nonumber \\
 \leq ||f||_{\Upsilon_T} \int_0^t\frac{Ce^{(d-1)(t-s)}}{ (t-s)^{\frac{1}{2}}}\mathcal{X}(s)ds.
\end{align} 
Equations \eqref{x-derInt}, \eqref{x1} with $k=1$, and the above equation provide 
\begin{align*}
||\nabla_x f(t)||_{L^1_1}&\leq \frac{Ce^{(d-1)t}}{ t^{\frac{1}{2}}}||f_0||_{L^1_1}+||f||_{\Upsilon_T}\int_0^t\frac{Ce^{(d-1)(t-s)}}{ (t-s)^{\frac{1}{2}}}  \mathcal{X}(s)ds\\
 & \leq \frac{Ce^{(d-1)t}}{ t^{\frac{1}{2}}}||f_0||_{L^1_1}+||f||_{\Upsilon_T}\int_0^t\frac{Ce^{(d-1)(T-s)}}{ (t-s)^{\frac{1}{2}}}  \mathcal{X}(s)ds.
\end{align*}
This estimate and \eqref{x33} imply
\begin{align*}
\mathcal{X}(t)\leq \frac{Ce^{(d-1)t}}{ t^{\frac{1}{2}}}\max\{||f_0||_{L^1_1}, ||f_0||_{L^k_1}\}+||f||_{\Upsilon_T}\int_0^t\frac{Ce^{(d-1)(T-s)}}{\min\{ (t-s)^{\frac{1}{2}}, \,(t-s)^{\frac{d}{k}+\frac{1}{2}}\}}  \mathcal{X}(s)ds.
\end{align*}
We apply \cite[Lemma 2.5]{Lin} and obtain
\begin{align*}
& \mathcal{X}(t)\leq \frac{Ce^{(d-1)t}}{ t^{\frac{1}{2}}}\max\{||f_0||_{L^1_1}, ||f_0||_{L^k_1}\}+||f||_{\Upsilon_T} \int_0^t\frac{Ce^{(d-1)T} \mathcal{X}(0)}{ s^{\frac{1}{2}} \min\{ (t-s)^{\frac{1}{2}}, \,(t-s)^{\frac{d}{k}+\frac{1}{2}}\}}  ds.
\end{align*}
 \cite[Lemma 3.13]{Her} shows that $$\int_0^t\frac{1}{ s^{\frac{1}{2}} \min\{ (t-s)^{\frac{1}{2}}, \,(t-s)^{\frac{d}{k}+\frac{1}{2}}\}}  ds\leq \frac{C}{t^{\frac{d}{k}}},\,\,\,\, \forall\, t\in (0,T].$$    
The above inequalities show
\begin{align}\label{x3}
 t^{\frac{1}{2}}\mathcal{X}(t)\leq Ce^{(d-1)t}\max\{||f_0||_{L^1_1}, ||f_0||_{L^k_1}\}+||f||_{\Upsilon_T} Ce^{(d-1)T} \mathcal{X}(0) t^{\frac{1}{2}-\frac{d}{k}}.
\end{align}
Because of   $\frac{1}{2}-\frac{d}{k}>0$ we conclude that $t^{\frac{1}{2}}\mathcal{X}(t)$ is bounded for all $t\in [0,T].$
\\[2mm]
{\textbf{Part 2: Estimates on $\nabla_p f$:}}
Next we take the gradient with respect to $p$ in the integral equation using \eqref{p_i-der} 
\begin{align}\label{p0}
& \nabla_p f(t,x,p)=\nabla_p \int_{\mathbb{R}^{2d}}G(t,x,p,x_0,p_0)f_0(x_0,p_0)dx_0dp_0\nonumber\\
&+2(1-e^t)\int_0^t\int_{\mathbb{R}^{2d}}f(s,x_0,p_0)\nabla_{x_0}f(s,x_0,p_0)p_0\cdot(\nabla_{p_0}-\nabla_{x_0})G(t-s,x,p,x_0,p_0)dx_0dp_0ds\nonumber\\
&+2e^t\int_0^t\int_{\mathbb{R}^{2d}}f(s,x_0,p_0)\nabla_{p_0}f(s,x_0,p_0)p_0\cdot(\nabla_{p_0}-\nabla_{x_0}) G(t-s,x,p,x_0,p_0)dx_0dp_0ds
\nonumber \\
&+de^t\int_0^t\int_{\mathbb{R}^{2d}} f^2(s,x_0,p_0)(\nabla_{p_0}-\nabla_{x_0})G(t-s,x,p,x_0,p_0) dx_0dp_0ds.
\end{align}
Let us denote 
$$\mathcal{Y}(t)\colonequals \max\{||\nabla_pf(t)||_{L^k_1}, ||\nabla_p f(t)||_{L^1_1}\}.$$
We estimate the $L^k_1$-norm of $|\nabla_pf(t)|$:
The first term on the right hand side of \eqref{p0} can be estimated using \eqref{p-der}
\begin{equation}\label{p1}
\left|\left| \nabla_p \int_{\mathbb{R}^{2d}} G(t,x,p,x_0,p_0)f_0(x_0,p_0)dx_0dp_0\right|\right|_{L^k_1}\leq \frac{Ce^{d t}}{ t^{\frac{1}{2}}}||f_0||_{L^k_1}.
\end{equation}
Equations \eqref{x2}, the fact that  $s^{\frac{1}{2}}||\nabla_x f(s)||_{L^k_1}$  is bounded by a constant because of \eqref{x3}, and \cite[Lemma 3.13]{Her} provide 
\begin{align}\label{p22}
2(e^t-1)& \left|\left|\int_0^t\int_{\mathbb{R}^{2d}}f(s,x_0,p_0)\nabla_{x_0}f(s,x_0,p_0)p_0\cdot(\nabla_{p_0}-\nabla_{x_0})G(t-s,x,p,x_0,p_0)dx_0dp_0ds \right|\right|_{L^k_1}\nonumber \\
 &\leq ||f||_{\Upsilon_T}(e^t-1) \int_0^t\frac{Ce^{d(1-\frac{1}{k}-\frac{1}{d})(t-s)}}{ (t-s)^{\frac{d}{k}+\frac{1}{2}}}  ||\nabla_x f(s)||_{L^k_1}ds\nonumber\\
&\leq ||f||_{\Upsilon_T}(e^t-1)\int_0^t\frac{Ce^{d(1-\frac{1}{k}-\frac{1}{d})(T-s)}}{ s^{\frac{1}{2}}(t-s)^{\frac{d}{k}+\frac{1}{2}}}  ds\nonumber \\
&\leq C||f||_{\Upsilon_T}(e^t-1)t^{-\frac{d}{k}}.
\end{align}
The third terms on the right hand side of \eqref{p0} can be estimated similarly to \eqref{x2}
\begin{align}\label{p2}
\left|\left|2e^t\int_0^t\int_{\mathbb{R}^{2d}}f(s,x_0,p_0)\nabla_{p_0}f(s,x_0,p_0)p_0\cdot(\nabla_{p_0}-\nabla_{x_0})G(t-s,x,p,x_0,p_0)dx_0dp_0ds \right|\right|_{L^k_1}\nonumber \\
 \leq ||f||_{\Upsilon_T}e^t\int_0^t\frac{Ce^{d(1-\frac{1}{k}-\frac{1}{d})(t-s)}}{ (t-s)^{\frac{d}{k}+\frac{1}{2}}}  ||\nabla_p f(s)||_{L^k_1}ds.
\end{align}
Using \eqref{fun.est} we estimate the last term on the right hand side of this equation
\begin{align*}
&\left|\left|de^t\int_0^t\int_{\mathbb{R}^{2d}} f^2(s,x_0,p_0)(\nabla_{p_0}-\nabla_{x_0})G(t-s,x,p,x_0,p_0) dx_0dp_0ds\right|\right|_{L^k_1}\\
&\leq  e^t\int_0^t\frac{C e^{(d-1)(t-s)}}{ (t-s)^{\frac{1}{2}}}||f^2(s)||_{L^k_1}ds.
\end{align*}
We can estimate $$||f^2(s)||_{L^k_1}\leq ||f(s)||_{L^{\infty}}||f(s)||_{L^k_1}\leq ||f||_{\Upsilon_T}^2 $$
and $$\int_0^t \frac{ e^{(d-1)(t-s)}}{ (t-s)^{\frac{1}{2}}}ds\leq e^{(d-1)t} \int_0^t\frac{1}{ (t-s)^{\frac{1}{2}}}ds=2e^{(d-1)t}t^{\frac{1}{2}}.$$
These estimates provide 
\begin{align}\label{p3}
\left|\left|de^t\int_0^t\int_{\mathbb{R}^{2d}} f^2(s,x_0,p_0)(\nabla_{p_0}-\nabla_{x_0})G(t-s,x,p,x_0,p_0) dx_0dp_0ds\right|\right|_{L^k_1}\leq  C||f||^2_{\Upsilon_T}e^{dt}t^{\frac{1}{2}}.
\end{align}
Equations \eqref{p0}-\eqref{p3} show 
\begin{align}\label{pL^k_m}
||\nabla_p f(t)||_{L^k_1}\leq & \frac{Ce^{d t}}{t^{\frac{1}{2}}}||f_0||_{L^k_m}+C||f||_{\Upsilon_T}(e^t-1)t^{-\frac{d}{k}}+C||f||^2_{\Upsilon_T}e^{dt}t^{\frac{1}{2}}\nonumber \\
& +||f||_{\Upsilon_T}e^T\int_0^t\frac{Ce^{d(1-\frac{1}{k}-\frac{1}{d})(T-s)}}{ (t-s)^{\frac{d}{k}+\frac{1}{2}}}  ||\nabla_p f(s)||_{L^k_1}ds.
\end{align}

Next we estimate the $L^1_1$-norm of $|\nabla_p f|:$ Equation \eqref{p1} holds with $k=1.$
Equation \eqref{xL^1_1} and the fact that $s^{\frac{1}{2}}\mathcal{X}(s)$ is bounded because of \eqref{x3}, and \cite[Lemma 3.13]{Her}  imply 
\begin{align*}
2(e^t-1)& \left|\left|\int_0^t\int_{\mathbb{R}^{2d}}f(s,x_0,p_0)\nabla_{x_0}f(s,x_0,p_0)p_0\cdot(\nabla_{p_0}-\nabla_{x_0})G(t-s,x,p,x_0,p_0)dx_0dp_0ds \right|\right|_{L^1_1}\nonumber \\
 &\leq ||f||_{\Upsilon_T}(e^t-1) \int_0^t\frac{Ce^{(d-1)(t-s)}}{ s^{\frac{1}{2}}(t-s)^{\frac{1}{2}}} ds\nonumber\\
&\leq C ||f||_{\Upsilon_T}(e^t-1).
\end{align*}
The third terms on the right hand side of \eqref{p0} can be estimated similarly to \eqref{xL^1_1}
\begin{align*}
\left|\left|2e^t\int_0^t\int_{\mathbb{R}^{2d}}f(s,x_0,p_0)\nabla_{p_0}f(s,x_0,p_0)p_0\cdot(\nabla_{p_0}-\nabla_{x_0})G(t-s,x,p,x_0,p_0)dx_0dp_0ds \right|\right|_{L^1_1}\nonumber \\
 \leq ||f||_{\Upsilon_T}e^t\int_0^t\frac{Ce^{(d-1)(t-s)}}{ (t-s)^{\frac{1}{2}}}  \mathcal{Y}(s)ds.
\end{align*}
Equation \eqref{p3} holds for $k=1$
\begin{align*}
\left|\left|de^t\int_0^t\int_{\mathbb{R}^{2d}} f^2(s,x_0,p_0)(\nabla_{p_0}-\nabla_{x_0})G(t-s,x,p,x_0,p_0) dx_0dp_0ds\right|\right|_{L^1_1}\leq  C||f||^2_{\Upsilon_T}e^{dt}t^{\frac{1}{2}}.
\end{align*}
Equation \eqref{p0} and the above estimates provide
\begin{align*}
||\nabla_p f(t)||_{L^1_1}\leq & \frac{Ce^{d t}}{t^{\frac{1}{2}}}||f_0||_{L^1_1}+C||f||_{\Upsilon_T}(e^t-1)+C||f||^2_{\Upsilon_T}e^{dt}t^{\frac{1}{2}}\\
& +||f||_{\Upsilon_T}e^T\int_0^t\frac{Ce^{(d-1)(T-s)}}{ (t-s)^{\frac{1}{2}}}  \mathcal{Y}(s)ds.
\end{align*}
This estimate and \eqref{pL^k_m} yield 
\begin{align*}
\mathcal{Y}(t)\leq & \frac{Ce^{d t}}{t^{\frac{1}{2}}}||f_0||_{\Upsilon}+C||f||_{\Upsilon_T}(e^t-1)t^{-\frac{d}{k}}+C||f||^2_{\Upsilon_T}e^{dt}t^{\frac{1}{2}}\\
& +||f||_{\Upsilon_T}e^T\int_0^t\frac{Ce^{(d-1)(T-s)}}{ (t-s)^{\frac{1}{2}}}  \mathcal{Y}(s)ds.
\end{align*}
We apply \cite[Lemma 2.5]{Lin} to obtain
\begin{align*}
&\mathcal{Y}(t)\leq  \frac{Ce^{d t}}{t^{\frac{1}{2}}}||f_0||_{\Upsilon}+C||f||_{\Upsilon_T}(e^t-1)t^{-\frac{d}{k}}+C||f||^2_{\Upsilon_T}e^{dt}t^{\frac{1}{2}}\\
& +||f||_{\Upsilon_T}e^T\int_0^t\frac{Ce^{(d-1)(T-s)}}{ (t-s)^{\frac{d}{k}+\frac{1}{2}}}  \left(\frac{e^{d s}}{s^{\frac{1}{2}}}||f_0||_{L^k_1}+C||f||_{\Upsilon_T}(e^s-1)s^{-\frac{d}{k}}+C||f||^2_{\Upsilon_T}e^{ds}s^{\frac{1}{2}}\right)ds.
\end{align*}
Using \cite[Lemma 3.13]{Her}  one can show that  
$$t^{\frac{1}{2}}\int_0^t\frac{Ce^{(d-1)(T-s)}}{ (t-s)^{\frac{d}{k}+\frac{1}{2}}}  \left(\frac{e^{d s}}{s^{\frac{1}{2}}}||f_0||_{L^k_1}+C||f||_{\Upsilon_T}(e^s-1)s^{-\frac{d}{k}}+C||f||^2_{\Upsilon_T}e^{ds}s^{\frac{1}{2}}\right)ds$$
is bounded by a positive constant. Therefore, we obtain that $t^{\frac{1}{2}}\mathcal{Y}(t)$ is bounded for all $t\in [0,T].$

\textbf{Part 3: Smoothness:} Equation \ref{Eq} is uniformly parabolic (with $\varepsilon\Delta_x$) with smooth coefficients, the solution $f(t,x,p)$ is in $L^{\infty}([0,T]\times\mathbb{R}^{2d} )$ for every $0<T<T_{\mathrm{max}}$, and  the nonlinearity is smooth and bounded (since $f$ is bounded).  Thus, the solution $f(t,x,p)$ is smooth for all $t>0$ and $(x,p)\in \mathbb{R}^{2d}$ by the regularity results from \cite[Chapter 5, Sections 8 and 10]{Lad} for  quasilinear parabolic equations. 
\end{proof}


\subsection{Mass conservation}
 Let $\zeta:[0,\infty)\to \mathbb{R}$ be a non-increasing function in $ C^{\infty}((0,\infty))$ such that $\zeta(s)=1$ for $s\in [0,1]$ and $\zeta(s)=0$ for $s\geq 2.$  We define
 \begin{equation}\label{zeta}
 \zeta_{\sigma}(x,p)\colonequals \zeta \big(\sigma \sqrt{|x|^2+|p|^2}\big) ,\, \, (x,p)\in \mathbb{R}^{2d}, \, \sigma>0.
 \end{equation}
 We note that $\zeta_{\sigma}(x, p)=1$ for $\sqrt{|x|^2+|p|^2}\leq 1/\sigma$ and $\zeta_{\sigma}(x,p)=0$ for $\sqrt{|x|^2+|p|^2}\geq 2/\sigma.$ There is a constant $C>0$ independent of $\sigma$ such that 
 \begin{align}\label{gr.z}
 |\nabla_{x,p}\zeta_{\sigma}(x,p)|&\leq C\sigma,\nonumber \\ |\Delta_{x,p} \zeta_{\sigma}(x, p)|&\leq C\sigma^2,\nonumber \\ \sqrt{|x|^2+|p|^2}(|\nabla_{x}\zeta_{\sigma}(x, p)|+|\nabla_{p}\zeta_{\sigma}(x,p)|)&\leq C,\nonumber \\ |\Delta_{x,p} \zeta_{\sigma}(x, p)|&\leq C
 \end{align}
  for all  $(x,p)\in \mathbb{R}^{2d}.$
\begin{lemma}\label{l:mass}
Let $k>2d$,   $\Upsilon \colonequals L^{\infty}(\mathbb{R}^{2d})\cap L^{1}_1(\mathbb{R}^{2d})\cap L^{k}_1(\mathbb{R}^{2d})$ with the norm in \eqref{norm}, and  $f_0\in \Upsilon$. Let $f_{\varepsilon}\in C([0,T_{\mathrm{max}}); \Upsilon)$ with $f_{\varepsilon}(0,x,p)=f_0(x,p)$   be the solution  to \eqref{Eq}  in Theorem \ref{th:loc}.   Then we have mass conservation
\begin{equation*}
\int_{\mathbb{R}^{2d}}f_{\varepsilon}(t,x,p)dpdx=\int_{\mathbb{R}^{2d}}f_0(x,p)dpdx, \, \, \, \, \forall\, t\in [0, T_{\mathrm{max}}).
\end{equation*}
\end{lemma}
\begin{proof}  We omit the dependence of the solution $f\colonequals f_{\varepsilon}$ on $\varepsilon$ for notational convenience. By Theorem \ref{th:reg} the solution $f(t,x,p)$ is smooth and satisfies \eqref{Eq} in the classical sense  for all $(x,p)\in \mathbb{R}^{2d}$ and $t\in (0, T_{\mathrm{max}})$.
We multiply \eqref{Eq} by $\zeta_{\sigma}$ and integrate by parts
\begin{align*}
\frac{d}{dt}\int_{\mathbb{R}^{2d}}\zeta_{\sigma}fdpdx=&\int_{\mathbb{R}^{2d}} (\Delta_p \zeta_{\sigma} f-\nabla_p \zeta_{\sigma}\cdot p(f- f^2))dpdx\\
&+\int_{\mathbb{R}^{2d}} (\varepsilon \Delta_x \zeta_{\sigma} f+\nabla_x \zeta_{\sigma}\cdot p(f- f^2))dpdx.
\end{align*}
We integrate this equation on the interval $[0,T]$ for any $T\in (0, T_{\mathrm{max}}):$
\begin{align*}
\int_{\mathbb{R}^{2d}}\zeta_{\sigma}f(T)dpdx-\int_{\mathbb{R}^{2d}}\zeta_{\sigma}f_0dpdx&=\int_0^T\int_{\mathbb{R}^{2d}} (\Delta_p \zeta_{\sigma} f-\nabla_p \zeta_{\sigma}\cdot p(f- f^2))dpdxdt\\
&+\int_0^T\int_{\mathbb{R}^{2d}} (\varepsilon \Delta_x \zeta_{\sigma} f+\nabla_x \zeta_{\sigma}\cdot p(f- f^2))dpdxdt.
\end{align*}
 The fact that  $\zeta_{\sigma}\to  1$ pointwise as $\sigma\to 0$  and  the Lebesgue dominated convergence theorem  show
 $$\lim_{\sigma\to 0}\big(\int_{\mathbb{R}^{2d}}\zeta_{\sigma}f(T)dpdx-\int_{\mathbb{R}^{2d}}\zeta_{\sigma}f(0)dpdx\big)=\int_{\mathbb{R}^{2d}}f(T)dpdx-\int_{\mathbb{R}^{2d}}f_0dpdx.$$ 
{The terms $\Delta_p \zeta_{\sigma} f-\nabla_p \zeta_{\sigma}\cdot p(f- f^2)$ and $\varepsilon \Delta_x \zeta_{\sigma} f+\nabla_x \zeta_{\sigma}\cdot p(f- f^2)$  converge  pointwise to zero  as $\sigma \to 0.$} Moreover, $\Delta_{x,p} \zeta_{\sigma} $,  $\nabla_x \zeta_{\sigma}\cdot p$, and $\nabla_p \zeta_{\sigma}\cdot p$ are bounded; see \eqref{gr.z}. Hence, we can apply the Lebesgue dominated convergence theorem 
 $$\lim_{\sigma\to 0}\int_0^T\int_{\mathbb{R}^{2d}} (\Delta_p \zeta_{\sigma} f-\nabla_p \zeta_{\sigma}\cdot p(f- f^2))dpdxdt=0$$
 and 
 $$\lim_{\sigma\to 0}\int_0^T\int_{\mathbb{R}^{2d}} (\varepsilon \Delta_x \zeta_{\sigma} f+\nabla_x \zeta_{\sigma}\cdot p(f- f^2))dpdxdt=0.$$
 These limits show
 $$\int_{\mathbb{R}^{2d}}f(T)dpdx=\int_{\mathbb{R}^{2d}}f_0dpdx$$
for all $T\in (0, T_{\mathrm{max}}).$
\end{proof}
\subsection{Contraction of the $L^1 $ norm and comparison principle}
We define
\begin{equation}\label{sign}
\text{sign}_{\tau}(s)\colonequals \begin{cases} -1,\, \, \,\, \, \,  &\text{if}\,\, \,s\leq -\tau\\
\eta_{\tau}(s),\, \, \, &\text{if}\, \, \, -\tau\leq s\leq  \tau\\
1, \, \, \, \, \, \, \,\, \,  \, \, &\text{if}\, \, \, s\geq \tau
\end{cases} 
\end{equation}
with an increasing odd function $\eta_{\tau}\in C^{\infty}([-\tau,\tau])$ such that $\text{sign}_{\tau}$ is $ C^{\infty}$ at $s=\pm \tau.$   Let $|\cdot|_{\tau}$ denote the primitive of $\text{sign}_{\tau}$ which vanishes at 0. One can  check that  $\text{sign}_{\tau}$ and $|\cdot|_{\tau}$ converge pointwise  to the usual sign  and absolute value functions as $\tau\to 0,$ respectively.

\begin{lemma}\label{lem:L1}
Let $k>2d$,   $\Upsilon \colonequals L^{\infty}(\mathbb{R}^{2d})\cap L^{1}_1(\mathbb{R}^{2d})\cap L^{k}_1(\mathbb{R}^{2d})$ with the norm in \eqref{norm}, and $f_0, g_0\in \Upsilon$.  Assume  $f_{\varepsilon},g_{\varepsilon}\in C([0,T_{\mathrm{max}}); \Upsilon)$ are solutions to \eqref{Eq} with  the initial conditions $f_0$ and $g_0$, respectively. Then 
\begin{equation*}
||f_{\varepsilon}(t)-g_{\varepsilon}(t)||_{L^{1}(\mathbb{R}^{2d})}\leq ||f_0-g_0||_{L^{1}(\mathbb{R}^{2d})}
\end{equation*}
for all $t\in [0, T_{\mathrm{max}}).$ In particular, if $f_0(x,p)\leq g_0(x,p)$ for almost every $(x,p)\in \mathbb{R}^{2d}, $ then $f_{\varepsilon}(t,x,p)\leq g_{\varepsilon}(t,x,p)$ for all $t\in [0,T_{\mathrm{max}})$ and for almost every $(x,p)\in \mathbb{R}^{2d}.$
\end{lemma}
\begin{proof}  We omit the dependence of the solutions $f\colonequals f_{\varepsilon}$  and $g\colonequals g_{\varepsilon}$ on $\varepsilon$ for notational convenience.

We compute   
\begin{equation*}
\partial_t(f-g)+p\cdot \nabla_x(f-g-(f^2-g^2))=\varepsilon \Delta_x(f-g)+\Delta_p(f-g)+\text{div}_p(p(f-g) -p (f^2-g^2)). 
\end{equation*}
 By Theorem \ref{th:reg} the solutions $f(t,x,p)$ and $g(t,x,p)$ are smooth and the above equation holds  for all $(x,p)\in \mathbb{R}^{2d}$ and $t\in (0, T_{\mathrm{max}})$. 
 Let $\zeta_{\sigma}$ be the function defined in \eqref{zeta}. We multiply this equation by $\zeta_{\sigma} \text{sign}_{\tau}(f-g)$ and integrate by parts
 \begin{align*}
\frac{d}{dt}& \int_{\mathbb{R}^{2d}}\zeta_{\sigma}|f-g|_{\tau} dpdx\\=&- \int_{\mathbb{R}^{2d}}\zeta_{\sigma}  \text{sign}_{\tau}'(f-g)|\nabla_p(f-g)|^2dpdx- \varepsilon\int_{\mathbb{R}^{2d}}\zeta_{\sigma}  \text{sign}_{\tau}'(f-g)|\nabla_x(f-g)|^2dpdx\\
&- \int_{\mathbb{R}^{2d}}\text{sign}_{\tau}(f-g) \nabla_x\zeta_{\sigma}\cdot (\varepsilon\nabla_x f-pf+pf^2 -\varepsilon\nabla_x  g+pg-pg^2)dpdx\\
&- \int_{\mathbb{R}^{2d}}\text{sign}_{\tau}(f-g) \nabla_p\zeta_{\sigma}\cdot (\nabla_p f+p(f- f^2)- \nabla_p g-p(g- g^2))dpdx\\
&- \int_{\mathbb{R}^{2d}}\zeta_{\sigma} \text{sign}_{\tau}'(f-g)p\cdot\nabla_p(f-g)(f-g)(1- f- g)dpdx\\
&+ \int_{\mathbb{R}^{2d}}\zeta_{\sigma} \text{sign}_{\tau}'(f-g)p\cdot\nabla_x(f-g)(f-g)(1- f- g)dpdx.
 \end{align*}
  We integrate this equation on the interval $[0,T],\, 0<T<T_{\mathrm{max}},$ and  use
 $\zeta_{\sigma}  \text{sign}_{\tau}'(f-g)|\nabla_p(f-g)|^2+\zeta_{\sigma}  \text{sign}_{\tau}'(f-g)|\nabla_x(f-g)|^2 \geq 0$ to estimate
  \begin{align*}
 \int_{\mathbb{R}^{2d}}&\zeta_{\sigma}|f(T)-g(T)|_{\tau} dpdx- \int_{\mathbb{R}^{2d}}\zeta_{\sigma}|f(0)-g(0)|_{\tau} dpdx\\ \leq &-\int_0^T  \int_{\mathbb{R}^{2d}}\text{sign}_{\tau}(f-g) \nabla_x\zeta_{\sigma}\cdot (\varepsilon\nabla_x f-pf+pf^2 -\varepsilon\nabla_x  g+pg-pg^2)dpdxds\\
&-\int_0^T \int_{\mathbb{R}^{2d}}\text{sign}_{\tau}(f-g) \nabla_p\zeta_{\sigma}\cdot (\nabla_p f+p(f- f^2)- \nabla_p g-p(g- g^2))dpdxds\\
&-\int_0^T \int_{\mathbb{R}^{2d}}\zeta_{\sigma} \text{sign}_{\tau}'(f-g)p\cdot\nabla_p(f-g)(f-g)(1- f- g)dpdxds\\
&+\int_0^T \int_{\mathbb{R}^{2d}}\zeta_{\sigma} \text{sign}_{\tau}'(f-g)p\cdot\nabla_x(f-g)(f-g)(1- f- g)dpdxds.
 \end{align*}
 Using the identities
 $$\text{sign}_{\tau}'(f-g)\nabla_x(f-g)(f-g)=\nabla_x((f-g)\text{sign}_{\tau}(f-g)-|f-g|_{\tau}),$$
 $$\text{sign}_{\tau}'(f-g)\nabla_p(f-g)(f-g)=\nabla_p((f-g)\text{sign}_{\tau}(f-g)-|f-g|_{\tau}),$$
 we integrate by parts
 \begin{align*}
 &\int_{\mathbb{R}^{2d}}\zeta_{\sigma}|f(T)-g(T)|_{\tau} dpdx- \int_{\mathbb{R}^{2d}}\zeta_{\sigma}|f(0)-g(0)|_{\tau} dpdx\\ \leq &-\int_0^T  \int_{\mathbb{R}^{2d}}\text{sign}_{\tau}(f-g) \nabla_x\zeta_{\sigma}\cdot (\varepsilon\nabla_x f-pf+pf^2 -\varepsilon\nabla_x  g+pg-pg^2)dpdxds\\
&-\int_0^T \int_{\mathbb{R}^{2d}}\text{sign}_{\tau}(f-g) \nabla_p\zeta_{\sigma}\cdot (\nabla_p f+p(f- f^2)- \nabla_p g-p(g- g^2))dpdxds\\
&+ \int_0^T\int_{\mathbb{R}^{2d}}((f-g)\text{sign}_{\tau}(f-g)-|f-g|_{\tau})\mathrm{div}_p(p\zeta_{\sigma} (1- f- g))dpdxdt\\
&-\int_0^T \int_{\mathbb{R}^{2d}}((f-g)\text{sign}_{\tau}(f-g)-|f-g|_{\tau})\mathrm{div}_x(p\zeta_{\sigma} (1- f- g))dpdxdt.
 \end{align*} 
 Since $|f-g|_{\tau}\to |f-g|,$ $\text{sign}_{\tau}(f-g)\to \text{sign}(f-g),$ and $(f-g)\text{sign}_{\tau}(f-g)-|f-g|_{\tau}\to 0$ pointwise  as $\tau\to 0,$ we conclude that 
\begin{align*}
 \int_{\mathbb{R}^{2d}}&\zeta_{\sigma}|f(T)-g(T)| dpdx- \int_{\mathbb{R}^{2d}}\zeta_{\sigma}|f(0)-g(0)| dpdx\\ \leq &-\int_0^T  \int_{\mathbb{R}^{2d}}\text{sign}(f-g) \nabla_x\zeta_{\sigma}\cdot (\varepsilon\nabla_x f-pf+pf^2 -\varepsilon\nabla_x  g+pg-pg^2)dpdxds\\
&-\int_0^T \int_{\mathbb{R}^{2d}}\text{sign}(f-g) \nabla_p\zeta_{\sigma}\cdot (\nabla_p f+p(f- f^2)- \nabla_p g-p(g- g^2))dpdxds.
 \end{align*}  
  We have that $|\varepsilon\nabla_x f-pf+pf^2 -\varepsilon\nabla_x  g+pg-pg^2|$ and $|\nabla_p f+p(f- f^2)- \nabla_p g-p(g- g^2)|$ belong to $ L^1([0,T]\times\mathbb{R}^{2d})$ by Theorems \ref{th:loc} and  \ref{th:reg},  $|\nabla_{x,p}\zeta_{\sigma}|\to 0$ uniformly as $\sigma\to 0$ by \eqref{gr.z}, and $|\text{sign}(f-g)|\leq 1$.  Thus, we can pass to the limit as  $\sigma \to 0$ 
  and obtain 
 $$ \int_{\mathbb{R}^{2d}}|f(T)-g(T)| dpdx\leq \int_{\mathbb{R}^{2d}}|f_0-g_0| dpdx.$$ Since $T\in (0,T_{\mathrm{max}})$ is arbitrary, we obtain the  claimed result. 
 
 Let $f_0\leq g_0 $ for almost every $(x,p)\in \mathbb{R}^{2d}. $ {The function $|f(t,x,p)-g(t,x,p)|+f(t,x,p)-g(t,x,p)$ is nonnegative almost everywhere. It is zero a.e.\,if  $f\leq g $ almost everywhere. } 
 The {$L^1$-contractivity} and the mass conservation  imply 
 \begin{align*}
 0&\leq  \int_{\mathbb{R}^{2d}}(|f(t,x,p)-g(t,x,p)|+f(t,x,p)-g(t,x,p)) dpdx \\&\leq  \int_{\mathbb{R}^{2d}}(|f_0(x,p)-g_0(x,p)|+f_0(x,p)-g_0(x,p)) dpdx\\
 &= \int_{\mathbb{R}^{2d}}(g_0-f_0+f_0-g_0) dpdx=0.
 \end{align*}
{We conclude that  $ \int_{\mathbb{R}^{2d}}(|f(t,x,p)-g(t,x,p)|+f(t,x,p)-g(t,x,p)) dpdx$ is zero for all $t\geq 0.$}
 This shows $|f(t,x,p)-g(t,x,p)|+f(t,x,p)-g(t,x,p)= 0$ for a.e.  $(x,p)$ and  for all $t\geq 0,$ which means $f(t,x,p)\leq g(t,x,p)$ for almost every $(x,p)\in \mathbb{R}^{2d}$ and  for all $t\in [0,T_{\mathrm{max}}).$
 
\end{proof}
By taking  $f_{\varepsilon}\equiv 0$ and $g_{\varepsilon}\equiv 1$ in Lemma \ref{lem:L1}, we obtain the following corollary.
\begin{corollary}\label{cor}
Let  $k>2d$,  $\Upsilon \colonequals L^{\infty}(\mathbb{R}^{2d})\cap L^{1}_1(\mathbb{R}^{2d})\cap L^{k}_1(\mathbb{R}^{2d})$ with the norm in \eqref{norm}, and $f_0\in \Upsilon$. Let $f_{\varepsilon}\in C([0,T_{\mathrm{max}}); \Upsilon)$ with $f_{\varepsilon}(0,x,p)=f_0(x,p) $   be the solution  to \eqref{Eq}  in Theorem \ref{th:loc}.   Assume additionally that $0\leq f_0(x,p)\leq 1$ for almost every $(x,p)\in \mathbb{R}^{2d}, $ then $0\leq f(t,x,p)\leq 1$ for all $t\in [0,T_{\mathrm{max}})$ and for almost every $(x,p)\in \mathbb{R}^{2d}.$ 
\end{corollary} 

\subsection{Bounds on the $p-$moments and global existence}
We establish time dependent bounds on the $p-$moments. More precisely, we will show that the moments increase at most as a polynomial of $t$.
\begin{lemma}\label{lem:p-moment}
Let   $k>2d$,  $\Upsilon \colonequals L^{\infty}(\mathbb{R}^{2d})\cap L^{1}_1(\mathbb{R}^{2d})\cap L^{k}_1(\mathbb{R}^{2d})$ with the norm in \eqref{norm}, and  $f_0\in \Upsilon$. Let $f_{\varepsilon}\in C([0,T_{\mathrm{max}}); \Upsilon)$ with $f_{\varepsilon}(0,x,p)=f_0(x,p)$  be the solution  to \eqref{Eq}  in Theorem \ref{th:loc}.  Assume additionally  $f_0\in L^1_{k}(\mathbb{R}^{2d})$ and $0\leq f_0(x,p)\leq 1$ almost every $(x,p)\in \mathbb{R}^{2d}.$   Then, for $t\in [0,T_{\mathrm{max}})$ and $1\leq \gamma\leq k,$ the $\gamma-$moment of $f$ with respect to $p$, i.e., $\int_{\mathbb{R}^{2d}}|p|^{\gamma}f(t,x,p)dxdp,$ is bounded by a polynomial $P_{\lceil \gamma/2 \rceil}(t)$ of degree $\lceil \gamma/2 \rceil,$ which is independent of $\varepsilon$ and depends only on  $\int_{\mathbb{R}^{2d}}f_0(x,p)dxdp,$ $\int_{\mathbb{R}^{2d}}|p|^{\gamma}f_0(x,p)dxdp,$ $\gamma$, and $d$. Here,  $\lceil\gamma/2 \rceil$ denotes the smallest integer  larger or equal to $\gamma/2$.    
\end{lemma}
\begin{proof}   We omit the dependence of the solution $f\colonequals f_{\varepsilon}$ on $\varepsilon$ for notational convenience. 
We will prove by induction on $\gamma$. First, we will show that the second moment is bounded. Afterwards, we will assume that we can bound the $(\gamma-2)$-moment and from the induction hypothesis  we obtain that the $\gamma$-moment is bounded.

 Let  $\phi:\mathbb{R}\to [0,\infty)$ be a smooth function such that $\phi(s)=\phi(-s),$ $\phi(s)=s^2$ for $|s|\leq 1,$ and $\phi(s)=4$ for $|s|\geq 2.$ We also require that $\phi'(s)\geq 0$ for $1\leq s\leq 2.$  Let $R>0$ and $\phi_R(p)\colonequals R^2\phi(|p|/R). $
 Let $\zeta_{\sigma}$ be the function defined in \eqref{zeta}.  We multiply \eqref{Eq} by $\phi_R\zeta_{\sigma}$ and integrate by parts
\begin{align*}
\frac{d}{dt}\int_{\mathbb{R}^{2d}}\zeta_{\sigma}\phi_R fdxdp=&\varepsilon \int_{\mathbb{R}^{2d}}\Delta_x\zeta_{\sigma}\phi_R fdxdp\\
&+\int_{\mathbb{R}^{2d}}(\Delta_p\zeta_{\sigma}\phi_R+2\nabla_p\zeta_{\sigma}\cdot \nabla_p \phi_R+\Delta_p \phi_R \zeta_{\sigma} ) fdxdp\\
&+ \int_{\mathbb{R}^{2d}}\nabla_x\zeta_{\sigma}\cdot p \phi_R (f-f^2)dxdp\\
&-\int_{\mathbb{R}^{2d}}(\nabla_p\zeta_{\sigma}\cdot p\phi_R+\zeta_{\sigma}\nabla_p \phi_R\cdot p) (f-f^2)dxdp.
\end{align*}
We integrate this equation on the interval $[0,T]$ with $0<T<T_{\mathrm{max}}$
\begin{align*}
\int_{\mathbb{R}^{2d}}&\zeta_{\sigma}\phi_R f(T)dxdp-\int_{\mathbb{R}^{2d}}\zeta_{\sigma}\phi_R f(0)dxdp\\
=&\varepsilon \int_0^T\int_{\mathbb{R}^{2d}}\Delta_x\zeta_{\sigma}\phi_R fdxdpdt\\
&+\int_0^T\int_{\mathbb{R}^{2d}}(\Delta_p\zeta_{\sigma}\phi_R+2\nabla_p\zeta_{\sigma}\cdot \nabla_p \phi_R+\Delta_p \phi_R \zeta_{\sigma} ) fdxdpdt\\
&+ \int_0^T\int_{\mathbb{R}^{2d}}\nabla_x\zeta_{\sigma}\cdot p \phi_R (f-f^2)dxdpdt\\
&-\int_0^T\int_{\mathbb{R}^{2d}}(\nabla_p\zeta_{\sigma}\cdot p\phi_R+\zeta_{\sigma}\nabla_p \phi_R\cdot p) (f-f^2)dxdpdt.
\end{align*}
Equation \eqref{gr.z} shows that  $\Delta_x\zeta_{\sigma}$, $\Delta_p\zeta_{\sigma}$,  $\nabla_x\zeta_{\sigma}$, and $\nabla_p\zeta_{\sigma}$ converge to zero uniformly as $\sigma \to 0$, and $\nabla_x\zeta_{\sigma}\cdot p$ and $\nabla_p\zeta_{\sigma}\cdot p$ are bounded.
Since $0\leq \phi_R\leq R^2$, $f\in \Upsilon_T,$ and $0\leq f\leq 1,$ we can pass to the limit as $\sigma \to 0$ by the Lebesgue dominated convergence theorem
\begin{align*}
\int_{\mathbb{R}^{2d}} \phi_R f(T)dxdp-\int_{\mathbb{R}^{2d}}\phi_R f_0dxdp
=&\int_0^T\int_{\mathbb{R}^{2d}} \Delta_p \phi_R   fdxdpdt-\int_0^T\int_{\mathbb{R}^{2d}}\nabla_p \phi_R\cdot p (f-f^2)dxdpdt.
\end{align*} 
 We have  $\nabla_p \phi_R\cdot p\geq 0$ and $ |\Delta_p \phi_R  |\leq C$ for some constant $C>0$ (independent of $R$).     Using the mass conservation and $f-f^2\geq 0$,  we can estimate
 \begin{align*}
\int_{\mathbb{R}^{2d}} \phi_R f(T)dxdp\leq \int_{\mathbb{R}^{2d}}\phi_R f_0dxdp
+CT\int_{\mathbb{R}^{2d}}    f_0dxdp.
\end{align*}
We now let $R\to \infty.$ Fatou's lemma provides
  \begin{align*}
\int_{\mathbb{R}^{2d}} |p|^2 f(T)dxdp\leq \int_{\mathbb{R}^{2d}}|p|^2 f_0dxdp
+CT\int_{\mathbb{R}^{2d}}    f_0dxdp.
\end{align*}
This show that the second moment $\int_{\mathbb{R}^{2d}} |p|^2 f(T)dxdp$ is finite for all $T\in (0,T_{\mathrm{max}})$  if $ \int_{\mathbb{R}^{2d}}|p|^2 f_0dxdp$ is finite. Having this information we can actually redo the above computations by replacing $\phi_R$ with $|p|^2$ and obtain  
\begin{align*}
\int_{\mathbb{R}^{2d}}|p|^2 f(t)dxdp\leq\int_{\mathbb{R}^{2d}}|p|^2 f_0dxdp+2d t\int_{\mathbb{R}^{2d}} f_0dxdp,\,\,\,\,\forall\, t\in [0,T_{\mathrm{max}}).
\end{align*}
This shows that the second $\int_{\mathbb{R}^{2d}}|p|^2 f(t)dxdp$ is bounded by a linear polynomial.  
In view of $2|p|\leq 1+|p|^2$, we have
$$2\int_{\mathbb{R}^{2d}}|p| f(t)dxdp\leq \int_{\mathbb{R}^{2d}} f_0dxdp+\int_{\mathbb{R}^{2d}}|p|^2 f(t)dxdp.$$
Therefore, $\int_{\mathbb{R}^{2d}}|p| f(t)dxdp$ is also bounded by a linear polynomial. Also,  note $\lceil 1/2\rceil =1.$ Hence, the claimed bound holds for  $\gamma=1,2.$

Next, we assume that $\gamma\in \mathbb{N}$ and  the $(\gamma-2)$-moment is bounded by a polynomial of degree $\lceil \gamma/2-1\rceil.$ As we did before, one can show that the $\gamma$-moment $\int_{\mathbb{R}^{2d}} |p|^{\gamma} f(T)dxdp$ is finite for all $T\in (0,T_{\mathrm{max}})$  if $ \int_{\mathbb{R}^{2d}}|p|^{\gamma} f_0dxdp$ is finite. We compute the $\gamma$-moment 
for any $T<T_{\mathrm{max}}$
\begin{align*}
\int_{\mathbb{R}^{2d}}\zeta_{\sigma}&|p|^{\gamma} f(T)dxdp=\int_{\mathbb{R}^{2d}}\zeta_{\sigma}|p|^{\gamma} f_0dxdp+\varepsilon \int_0^T\int_{\mathbb{R}^{2d}}\Delta_x\zeta_{\sigma}|p|^{\gamma} fdxdpdt
\\&+\int_0^T\int_{\mathbb{R}^{2d}}(\Delta_p\zeta_{\sigma}|p|^{\gamma}+2\gamma|p|^{\gamma-2}\nabla_p\zeta_{\sigma}\cdot p+\gamma(\gamma-2+d)|p|^{\gamma-2} \zeta_{\sigma} ) fdxdpdt\\
&+ \int_0^T\int_{\mathbb{R}^{2d}}\nabla_x\zeta_{\sigma}\cdot p |p|^{\gamma} (f-f^2)dxdpdt\\
&-\int_0^T\int_{\mathbb{R}^{2d}}(\nabla_p\zeta_{\sigma}\cdot p|p|^{\gamma}+\gamma\zeta_{\sigma}|p|^{\gamma}) (f-f^2)dxdpdt.
\end{align*}
Equation \eqref{gr.z} shows that  $\Delta_x\zeta_{\sigma}$, $\Delta_p\zeta_{\sigma}$,  $\nabla_x\zeta_{\sigma}$, and $\nabla_p\zeta_{\sigma}$ converge to zero uniformly as $\sigma \to 0$, and $\nabla_x\zeta_{\sigma}\cdot p$ and $\nabla_p\zeta_{\sigma}\cdot p$ are bounded.
Since $\int_{\mathbb{R}^{2d}} |p|^{\gamma} f(T)dxdp$ is finite and  $0\leq f\leq 1,$ we can pass to the limit as $\sigma \to 0$ by the Lebesgue dominated convergence theorem
\begin{align*}
\int_{\mathbb{R}^{2d}}|p|^{\gamma} f(T)dxdp=&\int_{\mathbb{R}^{2d}}|p|^{\gamma} f_0dxdp+
\gamma(\gamma-2+d)\int_0^T\int_{\mathbb{R}^{2d}}|p|^{\gamma-2}  fdxdpdt\\
&-\gamma\int_0^T\int_{\mathbb{R}^{2d}} |p|^{\gamma} (f-f^2)dxdpdt.
\end{align*}
Using $0\leq f\leq 1$ we get 
\begin{align}\label{Pgam}
\int_{\mathbb{R}^{2d}}|p|^{\gamma} f(T)dxdp\leq &\int_{\mathbb{R}^{2d}}|p|^{\gamma} f_0dxdp+
\gamma(\gamma-2+d)\int_0^T\int_{\mathbb{R}^{2d}}|p|^{\gamma-2}  fdxdpdt.
\end{align}
By induction, $\int_{\mathbb{R}^{2d}}|p|^{\gamma-2}  fdxdp$ is bounded by a polynomial $P_{\lceil \gamma/2-1\rceil}(t)$ of degree $\lceil\gamma/2-1\rceil=\lceil\gamma/2\rceil -1.$ Since its integral on $[0,T]$ gives a polynomial of degree $\lceil \gamma/2\rceil$ with respect to $T$, we obtain the claimed result for all natural numbers  $\gamma\leq k$.  

Let $\gamma$ be not a natural number and $\lceil \gamma\rceil\leq k$. 
 We write $\gamma=\lceil \gamma  \rceil  -\{\gamma\},$ where $0<\{\gamma\}<1.$ Using $|p|^{\gamma}\leq C(1+|p|^{\lceil \gamma \rceil}) $ we have  $$\int_{\mathbb{R}^{2d}}|p|^{\gamma} f(T)dxdp\leq C\int_{\mathbb{R}^{2d}} f_0dxdp+C\int_{\mathbb{R}^{2d}}|p|^{\lceil \gamma \rceil} f(T)dxdp.$$ 
The right hand side of this equation is bounded by a polynomial $P_{\lceil \lceil \gamma \rceil/2\rceil}$. As we have $\lceil \lceil \gamma \rceil/2\rceil=\lceil  \gamma /2\rceil$, we obtain the claimed result for all $\gamma\geq 1$ with $\lceil \gamma\rceil\leq k$. If $k$ is a natural number, the proof is complete.

If $k$ is not  a natural number, then for $\lceil k \rceil -1<\gamma\leq k$,  we have $\lceil \gamma \rceil =\lceil k \rceil +1>k$ and so the above argument cannot be used. In this case we use the fact that  \eqref{Pgam} holds even $\gamma$ is not a natural number.  As we have  $\lceil \gamma-2\rceil=\lceil \gamma\rceil -2 \leq k ,$ we can apply the above argument to the right hand side of  \eqref{Pgam}, implying the claimed result in this case.
\end{proof}
\begin{remark}\label{remark}
We note that $L^{1}_{k}(\mathbb{R}^{2d}) \subset L^1_{l}(\mathbb{R}^{2d})$ for all $l\in [1, k]$ because of $$1+|p|^{l}\leq C(1+|p|^{k})$$ for some constant $C>0$. In particular, $L^{1}_{k}(\mathbb{R}^{2d})\subset L^1_{1}(\mathbb{R}^{2d}).$ Moreover, we have 
 $L^1_{k}(\mathbb{R}^{2d})\cap L^{\infty}(\mathbb{R}^{2d})\subset L^k_{1}(\mathbb{R}^{2d})$ because of 
\begin{align*}
||f||_{L_1^k}=\left(\int_{\mathbb{R}^{2d}}(1+|p|)^k |f|^kdxdp\right)^{\frac{1}{k}} &\leq ||f||_{L^{\infty}}^{1-\frac{1}{k}} \left(\int_{\mathbb{R}^{2d}}(1+|p|)^k |f|dxdp\right)^{\frac{1}{k}}\\
& \leq C||f||_{L^{\infty}}^{1-\frac{1}{k}} ||f||_{L^1_{k}}^{\frac{1}{k}}.
\end{align*}
Thus, we have   $L^{1}_{k}(\mathbb{R}^{2d}) \cap L^{\infty}(\mathbb{R}^{2d}) \subset \Upsilon.$
\end{remark}

Our analysis above on the properties of the local solution $f_{\varepsilon}$ constructed in Theorem \ref{th:loc} show that this solutions is actually global:  Let  $f_0\in L^1_{k}(\mathbb{R}^{2d})$ and $0\leq f_0\leq 1$ for almost every $(x,p)\in \mathbb{R}^{2d}.$ Then we have $f_0\in \Upsilon$ by Remark \ref{remark}. 
By Theorem \ref{th:loc}, if the solution  $f_{\varepsilon}$ is not global, i.e., $T_{\mathrm{max}}<\infty,$ then at least one of the limits 
$$\lim_{t\nearrow T_{\mathrm{max}}} ||f_{\varepsilon}(t)||_{L^{\infty}},\,\,\,\,\lim_{t\nearrow T_{\mathrm{max}}} ||f_{\varepsilon}(t)||_{L^{1}_1}, \,\,\,\text{and}\,\,\,\lim_{t\nearrow T_{\mathrm{max}}} ||f_{\varepsilon}(t)||_{L^{k}_1}$$
is infinite.
 Corollary \ref{cor}    shows that  $0\leq f_{\varepsilon}\leq 1$ implying that $\lim_{t\nearrow T_{\mathrm{max}}} ||f_{\varepsilon}(t)||_{L^{\infty}}$ is bounded by 1.  The mass conservation in Lemma \ref{l:mass} and $f\geq 0$ provide $||f_{\varepsilon}(t)||_{L^1}=||f_0||_{L^1}$. This equation,  Lemma \ref{lem:p-moment}, and Remark \ref{remark} provide that $\lim_{t\nearrow T_{\mathrm{max}}} ||f_{\varepsilon}(t)||_{L^{1}_1}$ and $\lim_{t\nearrow T_{\mathrm{max}}} ||f_{\varepsilon}(t)||_{L^{k}_1}$ are also finite. This shows that $T_{\mathrm{max}}=\infty$. Hence, we have the following result.
\begin{theorem}\label{th:global}
Let $k>2d$,    $f_0\in L^1_{k}(\mathbb{R}^{2d})$, $0\leq f_0\leq 1$ for almost every $(x,p)\in \mathbb{R}^{2d},$ and $\Upsilon \colonequals L^{\infty}(\mathbb{R}^{2d})\cap L^{1}_1(\mathbb{R}^{2d})\cap L^{k}_1(\mathbb{R}^{2d})$ with the norm in \eqref{norm}. Then the Cauchy problem \eqref{Eq} with the initial condition $f_0$ has a unique mild solution $f\in C([0,\infty); \Upsilon)$ such that
\begin{itemize}
\item   $f_{\varepsilon}(t,\cdot,\cdot) \in L^{1}_{k}(\mathbb{R}^{2d}) $ for all $t\geq 0$;
\item  $0\leq f_{\varepsilon}(t,x,p)\leq 1$  for all $t\in [0, \infty)$ and for almost every $(x,p)\in \mathbb{R}^{2d}.$ 
\item  $ \displaystyle \int_{\mathbb{R}^{2d}}f_{\varepsilon}(t,x,p)dpdx=\int_{\mathbb{R}^{2d}}f_0(x,p)dpdx, \, \, \, \, \forall\, t\in [0, \infty);$
\item For all $\gamma\in [1, k],$
  the $\gamma$-moment of $f_{\varepsilon}$ with respect to $p$, i.e., $\int_{\mathbb{R}^{2d}}|p|^{\gamma}f_{\varepsilon}(t,x,p)dxdp,$ is bounded by a polynomial $P_{\lceil \gamma/2 \rceil}(t)$   (depending only on  $\int_{\mathbb{R}^{2d}}f_0(x,p)dxdp,$ $\int_{\mathbb{R}^{2d}}|p|^{\gamma}f_0(x,p)dxdp,$ $\gamma$, and $d$) of degree $\lceil \gamma/2 \rceil$ for almost every $t\in (0,\infty)$.    
\item   $t^{\frac{1}{2}}|\nabla_x f_{\varepsilon}| $ and  $t^{\frac{1}{2}}|\nabla_p f_{\varepsilon}| $ belong to $L^{\infty}((0,T); L^1_1(\mathbb{R}^{2d})\cap L^k_1(\mathbb{R}^{2d}))$ for all $T>0$, and $f_{\varepsilon}\in C^{\infty}((0,\infty)\times \mathbb{R}^{2d})$. 
\end{itemize}
\end{theorem}  
\section{Solution to the degenerate parabolic equation}
In this section we prove the main results, i.e., existence and uniqueness of entropy solution as in Definition \ref{def:gen.sol}. We first obtain estimates for the first order derivatives of the global solution $f_{\varepsilon}$    constructed for \eqref{Eq} in Theorem \ref{th:global}. We use these bounds to show that there is a sequence $\varepsilon_n$ converging to zero  such that $f_{\varepsilon_{n}}$  converges to a $BV$ function $f$ as $n\to \infty$ and this limit is a unique entropy solution of \eqref{DEq}.
\subsection{Estimates on the first order derivatives}  
\begin{lemma}\label{l:epsilon}
Let $k>2d$,   $f_0\in L^1_{k}(\mathbb{R}^{2d}) $, and $0\leq f_0(x,p)\leq 1$ for almost every $(x,p)\in \mathbb{R}^{2d}.$  Let $f_{\varepsilon}(t,x,p)$ be the global solution to  the Cauchy problem \eqref{Eq} in Theorem \ref{th:global}.  
\begin{itemize}
\item[(i)] For all $t\geq 0$
\begin{align*}
\frac{1}{2}||f_{\varepsilon}(t)||^2_{L^2(\mathbb{R}^{2d})}+\int_0^t  ||\nabla_p f_{\varepsilon}(s)||_{L^2(\mathbb{R}^{2d})}^2ds\leq & \frac{1+dt}{2} ||f_0||_{L^1(\mathbb{R}^{2d})}.
\end{align*}
\item[(ii)] Assume   $f_0\in W^{1,1}(\mathbb{R}^{2d})$. Then for all $t\geq 0$
\begin{equation*}
||\nabla_x f_{\varepsilon}(t)||_{L^1(\mathbb{R}^{2d})} \leq 
||\nabla_x f_{0}||_{L^1(\mathbb{R}^{2d})},
\end{equation*}
and
\begin{equation*}
||\nabla_p f_{\varepsilon}(t)||_{L^1(\mathbb{R}^{2d})} \leq \left(||\nabla_p f_{0}||_{L^1(\mathbb{R}^{2d})} +t||\nabla_x f_{0}||_{L^1(\mathbb{R}^{2d})} \right) e^t.
\end{equation*}
\item[(iii)]   Assume  $\mathcal{L}[f_0]\colonequals \mathrm{div}_{p}(\nabla_p f_0+p(f_0-f_0^2))-p\cdot \nabla_x (f_0-f_0^2)$ and $\Delta_x f_0$ belong to $ L^1(\mathbb{R}^{2d})$. Then for all $t\geq 0$
\begin{equation*}
||\partial_t f_{\varepsilon}(t)||_{L^1(\mathbb{R}^{2d})} \leq  
||L f_{0}||_{L^1(\mathbb{R}^{2d})}
+\varepsilon||\Delta_x f_{0}||_{L^1(\mathbb{R}^{2d})}.
\end{equation*}
\end{itemize}
 
\end{lemma}
\begin{proof}
 We omit the dependence of the solution $f\colonequals f_{\varepsilon}$ on $\varepsilon$ for notational convenience.

\textbf{Part (i):}  Let $\zeta_{\sigma}$ be the function defined in \eqref{zeta}. We multiply \eqref{Eq} by $\zeta_{\sigma}f$ and integrate by parts
\begin{align*}
\frac{1}{2}\frac{d}{dt}\int_{\mathbb{R}^{2d}}\zeta_{\sigma}  f^2dxdp=&-\varepsilon\int_{\mathbb{R}^{2d}} \zeta_{\sigma} |\nabla_x f|^2dxdp+\frac{\varepsilon}{2}\int_{\mathbb{R}^{2d}} \Delta_x\zeta_{\sigma}  f^2dxdp\\
&+\int_{\mathbb{R}^{2d}}\nabla_x \zeta_{\sigma}\cdot p f^2\Big(\frac{1}{2}-\frac{2f}{3}\Big)dxdp\\
&-\int_{\mathbb{R}^{2d}} \zeta_{\sigma} |\nabla_p f|^2dxdp+ \frac{1}{2}\int_{\mathbb{R}^{2d}} \Delta_p\zeta_{\sigma}  f^2dxdp\\
&-\int_{\mathbb{R}^{2d}}\nabla_p \zeta_{\sigma}\cdot p f^2\Big(\frac{1}{2}-\frac{2f}{3}\Big)dxdp+d\int_{\mathbb{R}^{2d}} \zeta_{\sigma} f^2\Big(\frac{1}{2}-\frac{f}{3}\Big) dxdp.
\end{align*} 
 We integrate this equation on the interval $[0,T]$ with $T>0$ and use $\varepsilon\int_{\mathbb{R}^{2d}} \zeta_{\sigma} |\nabla_x f|^2dxdp\geq 0$
 \begin{align*}
\frac{1}{2}\int_{\mathbb{R}^{2d}}\zeta_{\sigma}  f^2(T)dxdp\leq & \frac{1}{2}\int_{\mathbb{R}^{2d}}\zeta_{\sigma}  f^2_0dxdp-\frac{\varepsilon}{2}\int_0^T\int_{\mathbb{R}^{2d}} \Delta_x\zeta_{\sigma}  f^2dxdpdt\\
&+\int_0^T\int_{\mathbb{R}^{2d}}\nabla_x \zeta_{\sigma}\cdot p f^2\Big(\frac{1}{2}-\frac{2f}{3}\Big)dxdpdt\\
&-\int_0^T\int_{\mathbb{R}^{2d}} \zeta_{\sigma} |\nabla_p f|^2dxdpdt+ \frac{1}{2}\int_0^T\int_{\mathbb{R}^{2d}} \Delta_p\zeta_{\sigma}  f^2dxdpdt\\
&-\int_0^T\int_{\mathbb{R}^{2d}}\nabla_p \zeta_{\sigma}\cdot p f^2\Big(\frac{1}{2}-\frac{2f}{3}\Big)dxdpdt+d\int_0^T\int_{\mathbb{R}^{2d}} \zeta_{\sigma} f^2\Big(\frac{1}{2}-\frac{f}{3}\Big)dxdpdt.
\end{align*}
The functions $\Delta_x\zeta_{\sigma} $, $\Delta_x\zeta_{\sigma} $, $\nabla_x \zeta_{\sigma}$, and $\nabla_p \zeta_{\sigma}$ converge to zero uniformly by \eqref{gr.z}. Since $f$ is bounded and in $(1+|p|)f\in L^{1}([0,T]\times \mathbb{R}^{2d})$, we can pass to the limit as $\sigma\to 0$
\begin{align*}
\frac{1}{2}\int_{\mathbb{R}^{2d}}  f^2(T)dxdp\leq & \frac{1}{2}\int_{\mathbb{R}^{2d}}  f^2_0dxdp
-\int_0^T\int_{\mathbb{R}^{2d}}  |\nabla_p f|^2dxdpdt+d\int_0^T\int_{\mathbb{R}^{2d}}  f^2\Big(\frac{1}{2}-\frac{f}{3}\Big)dxdpdt.
\end{align*}
Since $0\leq f_0, f\leq 1$ and the mass is conserved, we can estimate $$\frac{1}{2}\int_{\mathbb{R}^{2d}}  f^2_0dxdp\leq \frac{1}{2}\int_{\mathbb{R}^{2d}}  f_0dxdp$$
and 
$$d\int_0^T\int_{\mathbb{R}^{2d}}  f^2\Big(\frac{1}{2}-\frac{f}{3}\Big) dxdpdt\leq \frac{d}{2}\int_0^T\int_{\mathbb{R}^{2d}}  fdxdpdt\leq \frac{dT}{2}\int_{\mathbb{R}^{2d}}  f_0dxdp.$$
These estimates provide
\begin{align*}
\frac{1}{2}\int_{\mathbb{R}^{2d}}  f^2(T)dxdp+\int_0^T\int_{\mathbb{R}^{2d}}  |\nabla_p f|^2dxdpdt\leq  \frac{1+dT}{2}\int_{\mathbb{R}^{2d}}  f_0dxdp.
\end{align*}
Since $T>0$ is arbitrary, we obtain the  claimed result.

\textbf{Part (ii):} We first estimate $||\nabla_x f_{\varepsilon}(t)||_{L^1(\mathbb{R}^{2d})} $.
Since the solution $f(t,x,p)$ is smooth for $t>0$, its partial derivatives $u_i\colonequals \partial_{x_i}f$ for $i\in \{1,..., d\}$ satisfy
\begin{equation}\label{u_i}
\partial_t u_i+p\cdot \nabla_x ((1-2f)u_i)=\mathrm{div}_{p}(\nabla_p u_i+p(1-2f)u_i)+\varepsilon \Delta_x u_i.
\end{equation}
Let $\zeta_{\sigma}$ and $\text{sign}_{\tau}$ be the functions defined in \eqref{zeta} and \eqref{sign}, respectively. We multiply the above equation by $\zeta_{\sigma}\text{sign}_{\tau}(u_i)$ and integrate by parts
\begin{align*}
\frac{d}{dt}& \int_{\mathbb{R}^{2d}}\zeta_{\sigma}|u_i|_{\tau} dpdx\\=&- \int_{\mathbb{R}^{2d}}\zeta_{\sigma}  \text{sign}_{\tau}'(u_i)|\nabla_pu_i|^2dpdx- \varepsilon\int_{\mathbb{R}^{2d}}\zeta_{\sigma}  \text{sign}_{\tau}'(u_i)|\nabla_x u_i|^2dpdx\\
&+\varepsilon \int_{\mathbb{R}^{2d}}(\text{sign}_{\tau}'(u_i) \nabla_x\zeta_{\sigma}\cdot \nabla_xu_i u_i+\text{sign}_{\tau}(u_i)\Delta_x  \zeta_{\sigma} u_i)  dpdx
\\&
+\int_{\mathbb{R}^{2d}}(\text{sign}_{\tau}'(u_i) \nabla_p\zeta_{\sigma}\cdot \nabla_pu_i u_i+\text{sign}_{\tau}(u_i)\Delta_p \zeta_{\sigma} u_i)dpdx\\
&+ \int_{\mathbb{R}^{2d}}\text{sign}_{\tau}(u_i) \nabla_x\zeta_{\sigma}\cdot  p(1-2f)u_i dpdx- \int_{\mathbb{R}^{2d}}\text{sign}_{\tau}(u_i) \nabla_p\zeta_{\sigma}\cdot  p(1-2f)u_i dpdx\\
&- \int_{\mathbb{R}^{2d}}\zeta_{\sigma} \text{sign}_{\tau}'(u_i)p\cdot\nabla_pu_i(1- 2f)u_idpdx+ \int_{\mathbb{R}^{2d}}\zeta_{\sigma} \text{sign}_{\tau}'(u_i)p\cdot\nabla_x u_i (1- 2f)u_idpdx.
 \end{align*}
  We integrate this equation on the interval $[0,T],\, T>0,$ and  use
 $\zeta_{\sigma}  \text{sign}_{\tau}'(u_i)|\nabla_pu_i|^2\geq 0$ and $\zeta_{\sigma}  \text{sign}_{\tau}'(u_i)|\nabla_xu_i|^2 \geq 0$ to estimate
 \begin{align*}
 &\int_{\mathbb{R}^{2d}}\zeta_{\sigma}|u_i(T)|_{\tau} dpdx-\int_{\mathbb{R}^{2d}}\zeta_{\sigma}|u_i(0)|_{\tau} dpdx\\
&\leq \varepsilon \int_0^T\int_{\mathbb{R}^{2d}}(\text{sign}_{\tau}'(u_i) \nabla_x\zeta_{\sigma}\cdot \nabla_xu_i u_i+\text{sign}_{\tau}(u_i)\Delta_x  \zeta_{\sigma} u_i)  dpdxdt
\\&
+\int_0^T\int_{\mathbb{R}^{2d}}(\text{sign}_{\tau}'(u_i) \nabla_p\zeta_{\sigma}\cdot \nabla_pu_i u_i+\text{sign}_{\tau}(u_i)\Delta_p \zeta_{\sigma} u_i)dpdxdt\\
&+ \int_0^T\int_{\mathbb{R}^{2d}}\text{sign}_{\tau}(u_i) \nabla_x\zeta_{\sigma}\cdot  p(1-2f)u_i dpdxdt- \int_0^T\int_{\mathbb{R}^{2d}}\text{sign}_{\tau}(u_i) \nabla_p\zeta_{\sigma}\cdot  p(1-2f)u_i dpdxdt\\
&- \int_0^T\int_{\mathbb{R}^{2d}}\zeta_{\sigma} \text{sign}_{\tau}'(u_i)p\cdot\nabla_pu_i(1- 2f)u_idpdxdt+ \int_0^T\int_{\mathbb{R}^{2d}}\zeta_{\sigma} \text{sign}_{\tau}'(u_i)p\cdot\nabla_x u_i (1- 2f)u_idpdxdt.
 \end{align*}
 Using the identities
 $$\text{sign}_{\tau}'(u_i)\nabla_xu_iu_i =\nabla_x(u_i\text{sign}_{\tau}(u_i)-|u_i|_{\tau}),$$
 $$\text{sign}_{\tau}'(u_i)\nabla_pu_iu_i=\nabla_p(u_i\text{sign}_{\tau}(u_i)-|u_i|_{\tau}),$$
 we integrate by parts
\begin{align*}
& \int_{\mathbb{R}^{2d}}\zeta_{\sigma}|u_i(T)|_{\tau} dpdx-\int_{\mathbb{R}^{2d}}\zeta_{\sigma}|u_i(0)|_{\tau} dpdx\\
&\leq \varepsilon \int_0^T\int_{\mathbb{R}^{2d}} \Delta_x\zeta_{\sigma} |u_i|_{\tau}  dpdxdt
+\int_0^T\int_{\mathbb{R}^{2d}} \Delta_p\zeta_{\sigma}|u_i|_{\tau} dpdxdt\\
&+ \int_0^T\int_{\mathbb{R}^{2d}}\text{sign}_{\tau}(u_i) \nabla_x\zeta_{\sigma}\cdot  p(1-2f)u_i dpdxdt- \int_0^T\int_{\mathbb{R}^{2d}}\text{sign}_{\tau}(u_i) \nabla_p\zeta_{\sigma}\cdot  p(1-2f)u_i dpdxdt\\
&+ \int_0^T\int_{\mathbb{R}^{2d}}\mathrm{div}_p(p\zeta_{\sigma} (1- 2f))(u_i\text{sign}_{\tau}(u_i)-|u_i|_{\tau})dpdxdt\\
&- \int_0^T\int_{\mathbb{R}^{2d}}\mathrm{div}_x(p\zeta_{\sigma}  (1- 2f))(u_i\text{sign}_{\tau}(u_i)-|u_i|_{\tau})dpdxdt.
 \end{align*}
 Since $|u_i|_{\tau}\to |u_i|,$ $\text{sign}_{\tau}(u_i)\to \text{sign}(u_i),$ and $u_i\text{sign}_{\tau}(u_i)-|u_i|_{\tau}\to 0$ pointwise  as $\tau\to 0,$ we conclude that 
\begin{align*}
& \int_{\mathbb{R}^{2d}}\zeta_{\sigma}|u_i(T)| dpdx-\int_{\mathbb{R}^{2d}}\zeta_{\sigma}|u_i(0)| dpdx\\
&\leq \varepsilon \int_0^T\int_{\mathbb{R}^{2d}} \Delta_x\zeta_{\sigma} |u_i|  dpdxdt
+\int_0^T\int_{\mathbb{R}^{2d}} \Delta_p\zeta_{\sigma}|u_i| dpdxdt\\
&+ \int_0^T\int_{\mathbb{R}^{2d}}\text{sign}(u_i) \nabla_x\zeta_{\sigma}\cdot  p(1-2f)u_i dpdxdt- \int_0^T\int_{\mathbb{R}^{2d}}\text{sign}(u_i) \nabla_p\zeta_{\sigma}\cdot  p(1-2f)u_i dpdxdt.
 \end{align*}
The terms   $|\nabla_{x,p}\zeta_{\sigma}|$ and $|\Delta_{x,p} \zeta_{\sigma}|$ converge to zero uniformly as $\sigma\to 0$ by \eqref{gr.z}. We have $|1-2f|\leq 1$  and $ (1+|p|)u_i\in L^1([0,T]\times \mathbb{R}^{2d}) $ by Theorem \ref{th:reg}. Hence,     we can pass to the limit as  $\sigma \to 0$ 
  and obtain 
 $$\int_{\mathbb{R}^{2d}}|u_i(T)| dpdx\leq \int_{\mathbb{R}^{2d}}|u_i(0)| dpdx.$$ 
 Since $T>0$ is arbitrary, we obtain the  claimed result.
 
Next, we estimate $||\nabla_p f_{\varepsilon}(t)||_{L^1(\mathbb{R}^{2d})} $: The partial derivatives $v_i\colonequals \partial_{p_i}f$ for $i\in \{1,...,d\}$ satisfy 
 \begin{equation*}
\partial_t v_i+p\cdot \nabla_x ((1-2f)v_i)=\mathrm{div}_{p}(\nabla_p v_i+p(1-2f)v_i)+\varepsilon \Delta_x v_i+(1-2f)u_i+(1-2f)v_i.
\end{equation*}
This equation differs from  \eqref{u_i} only by the term $(1-2f)u_i+(1-2f)v_i.$ Hence, we repeat the above computations  by only caring this term and obtain  $$\int_{\mathbb{R}^{2d}}|v_i(T)| dpdx- \int_{\mathbb{R}^{2d}}|v_i(0)| dpdx\leq \int^T_0 \int_{\mathbb{R}^{2d}}\mathrm{sign}(v_i)(1-2f)u_idxdpdt+\int^T_0\int_{\mathbb{R}^{2d}} |v_i|(1-2f)dxdpdt$$
for all $T>0.$  Using $|\mathrm{sign}(v_i)|\leq 1$, $|1-2f|\leq 1$, and $\int_{\mathbb{R}^{2d}}|u_i(t)| dpdx\leq \int_{\mathbb{R}^{2d}}|u_i(0)| dpdx$ for all $t\geq 0,$ we estimate 
$$\int_{\mathbb{R}^{2d}}|v_i(T)| dpdx- \int_{\mathbb{R}^{2d}}|v_i(0)| dpdx\leq  T\int_{\mathbb{R}^{2d}}|u_i(0)|dxdp+\int^T_0 \int_{\mathbb{R}^{2d}}|v_i| dxdpdt.$$
Grönwall's inequality implies 
$$\int_{\mathbb{R}^{2d}}|v_i(T)| dpdx\leq \left(\int_{\mathbb{R}^{2d}}|v_i(0)| dpdx+T\int_{\mathbb{R}^{2d}}|u_i(0)|dxdp\right) e^T.$$
Since $T>0$ is arbitrary, we obtain the  claimed result.

\textbf{Part (iii):} The time derivative $ \partial_tf$ satisfies 
\begin{equation*}
\partial_t(\partial_tf)+p\cdot \nabla_x ((1-2f)(\partial_tf))=\mathrm{div}_{p}(\nabla_p (\partial_tf)+p(1-2f)(\partial_tf))+\varepsilon \Delta_x (\partial_tf).
\end{equation*}
This equation is the exactly same as \eqref{u_i}, for which we have proved in Part (ii) that
$$||\partial_tf(t)||_{L^1(\mathbb{R}^{2d})}\leq ||\partial_tf (0)||_{L^1(\mathbb{R}^{2d})},\,\,\,\,\forall\, t\geq 0. $$ 
The assumptions on $f_0$ and  $\partial_tf (0)=\mathcal{L}[f_0]+\varepsilon\Delta_x f_0$ yield the claimed estimate. 
\end{proof}
\begin{lemma}
Let $k>2d$,  $f_0\in L^1_{k}(\mathbb{R}^{2d}) $, and $0\leq f_0\leq 1$ for almost every $(x,p)\in \mathbb{R}^{2d}.$  Let $f_{\varepsilon}(t,x,p)$ be the global solution to  the Cauchy problem \eqref{Eq} in Theorem \ref{th:global}. 
For any smooth compactly supported $\phi\in C_c^{\infty}((0, \infty)\times \mathbb{R}^{2d}) $ with $\phi\geq 0$ and any constant $a\in \mathbb{R}$ 
\begin{align}\label{epdef}
\int_0^{\infty}\int_{\mathbb{R}^{2d}}\mathrm{sign}(f_{\varepsilon}-a)&\big\{(f_{\varepsilon}-a)\partial_t \phi +(f_{\varepsilon}-f^2_{\varepsilon}-(a-a^2))(p\cdot \nabla_x \phi-p\cdot \nabla_p \phi)\nonumber \\
&-\nabla_p f_{\varepsilon}\cdot\nabla_p \phi+\varepsilon (f_{\varepsilon}-a) \Delta_x \phi+d(a-a^2)\phi \big\}dxdpdt\geq 0.
\end{align}
\end{lemma}
\begin{proof}  We omit the dependence of the solution $f\colonequals f_{\varepsilon}$ on $\varepsilon$ for notational convenience.

 It is easy to check that $f-a$ satisfies
\begin{equation*}
\partial_t (f-a)+p\cdot \nabla_x (f-f^2-(a-a^2))=\mathrm{div}_{p}(\nabla_p f+p(f-f^2-(a-a^2)))+\varepsilon \Delta_x (f-a)+d(a-a^2).
\end{equation*}
Let $\text{sign}_{\tau}$ be the function defined in \eqref{sign}. We multiply the above equation by $\phi  \text{sign}_{\tau}(f-a)$ and integrate by parts
\begin{align*}
\int_0^{\infty} \int_{\mathbb{R}^{2d}}&\Big( \partial_t \phi |f-a|_{\tau}+\text{sign}_{\tau}(f-a)(p\cdot \nabla_x \phi  -p\cdot \nabla_p \phi)(f-f^2-(a-a^2))
\\&-\text{sign}_{\tau}(f-a) \nabla_p f\cdot \nabla_p \phi -\varepsilon \text{sign}_{\tau}(f-a) \nabla_x f\cdot \nabla_x \phi +d(a-a^2)\phi \text{sign}_{\tau}(f-a)\\
&+\phi  \text{sign}_{\tau}'(f-a)(p\cdot \nabla_x (f-a)-p\cdot \nabla_p (f-a)) (f-f^2-(a-a^2))\\
&-\phi  \text{sign}_{\tau}'(f-a)(|\nabla_pf|^2+ \varepsilon |\nabla_x f|^2)\Big)dxdpdt=0.
\end{align*}
Since $\text{sign}_{\tau}(f-a) \nabla_x f=\nabla_x |f-a|_{\tau}$ and  $\phi  \text{sign}_{\tau}'(f-a)(|\nabla_pf|^2+ \varepsilon |\nabla_x f|^2)$ is nonnegative, we have 
\begin{align*}
\int_0^{\infty} \int_{\mathbb{R}^{2d}}&\Big( \partial_t \phi |f-a|_{\tau}+\text{sign}_{\tau}(f-a)(p\cdot \nabla_x \phi  -p\cdot \nabla_p \phi)(f-f^2-(a-a^2))
\\&-\text{sign}_{\tau}(f-a) \nabla_p f\cdot \nabla_p \phi +\varepsilon |f-a|_{\tau}  \Delta_x \phi +d(a-a^2)\phi \text{sign}_{\tau}(f-a)\\
&+\phi  \text{sign}_{\tau}'(f-a)(p\cdot \nabla_x (f-a)-p\cdot \nabla_p (f-a)) (f-f^2-(a-a^2))\geq 0.
\end{align*}
Using the identities
 $$\text{sign}_{\tau}'(f-a)\nabla_x(f-a)(f-a)=\nabla_x((f-a)\text{sign}_{\tau}(f-a)-|f-a|_{\tau}),$$
 $$\text{sign}_{\tau}'(f-a)\nabla_p(f-a)(f-a)=\nabla_p((f-a)\text{sign}_{\tau}(f-a)-|f-a|_{\tau}),$$
 we integrate by parts
\begin{align*}
\int_0^{\infty}& \int_{\mathbb{R}^{2d}}\Big( \partial_t \phi |f-a|_{\tau}+\text{sign}_{\tau}(f-a)(p\cdot \nabla_x \phi  -p\cdot \nabla_p \phi)(f-f^2-(a-a^2))
\nonumber\\&-\text{sign}_{\tau}(f-a) \nabla_p f\cdot \nabla_p \phi +\varepsilon |f-a|_{\tau}  \Delta_x \phi +d(a-a^2)\phi \text{sign}_{\tau}(f-a)\nonumber \\
& -((f-a)\text{sign}_{\tau}(f-a)-|f-a|_{\tau})(\mathrm{div}_x(p\phi(1-f-a))-\mathrm{div}_p( p\phi (1-f-a))\Big)dxdpdt\geq 0.
\end{align*}
Since $|f-a|_{\tau}\to |f-a|$, $\text{sign}_{\tau}(f-a)\to \text{sign}(f-a)$, and $(f-a)\text{sign}_{\tau}(f-a)-|f-a|_{\tau}\to 0$  pointwise  as $\tau \to 0, $ we get by the Lebesgue dominated convergence theorem 
\begin{align*}
\int_0^{\infty}& \int_{\mathbb{R}^{2d}}\Big( \partial_t \phi |f-a|+\text{sign}(f-a)(p\cdot \nabla_x \phi  -p\cdot \nabla_p \phi)(f-f^2-(a-a^2))
\\&-\text{sign}(f-a) \nabla_p f\cdot \nabla_p \phi +\varepsilon |f-a|  \Delta_x \phi +d(a-a^2)\phi \text{sign} (f-a)\geq 0,
\end{align*}
which is the claimed inequality.

\end{proof}

\subsection{Passing to the limit $\varepsilon\to 0$}

\begin{proof}[\textbf{Proof of Theorem \ref{th:global existence}}]
By Theorem \ref{th:global} and Lemma \ref{l:epsilon} 
for every $T>0$, 
$$\int_0^T \int_{\mathbb{R}^{2d}} \big(f_{\varepsilon}(t,x,p)+ |\partial_t f_{\varepsilon}(t,x,p)|+ |\nabla_x f_{\varepsilon}(t,x,p)|+ |\nabla_p f_{\varepsilon}(t,x,p)|\big)dxdpdt$$ is uniformly bounded with respect to $\varepsilon\in (0, 1].$  This fact and  \cite[Proposition 9.3]{Bre} imply 
$$\lim_{s\searrow 0,\, |y|\to 0,\,|p|\to 0}\sup_{\varepsilon\in (0,1]}\int_0^T\int_{\mathbb{R}^{2d}}|f_{\varepsilon}(t+s,x+y, p+q)-f_{\varepsilon}(t,x,p)|dxdpdt=0.$$ Then, by the Kolmogorov-Riesz-Fréchet theorem \cite[Theorem 4.26]{Bre} and a diagonal argument, there exists a sequence $\varepsilon_n\to 0$ such that $ f_{\varepsilon_n}$ converges to a function $f$ in $L^1_{\mathrm{loc}}((0,\infty)\times \mathbb{R}^{2d})$ as $n\to \infty.$ 

We show that the limit $f$ is an entropy solution to \eqref{DEq}, i.e., we show $f$ satisfies the conditions in Definition \ref{def:gen.sol}: We may assume that $f_{\varepsilon_n}$ converges to $f$ pointwise almost everywhere in $(0,\infty)\times \mathbb{R}^{2d}$. This implies 
$$0\leq f\leq 1$$
almost everywhere in $(0,\infty)\times \mathbb{R}^{2d}$.

  The pointwise convergence and Fatou's lemma imply $$f \in L^1([0,T]\times \mathbb{R}^{2d})$$ for  every $T>0$.

 Next, we show that $f$ is in $BV((0,T)\times \mathbb{R}^{2d})$ for every $T>0$: Let $\phi\in [C^1_c((0,T)\times \mathbb{R}^{2d})]^{2d+1}$ and $||\phi||_{L^{\infty}([0,T]\times \mathbb{R}^{2d})}\leq 1. $ Since $\phi$ has a compact support and $ f_{\varepsilon_n}\to f$ in $L^1_{\mathrm{loc}}((0,\infty)\times \mathbb{R}^{2d}),$ we have
\begin{align*}
\left|\int_0^{T}\int_{\mathbb{R}^{2d}}f\mathrm{div}_{t,x,p}\phi dxdpdt\right|&=\left|\lim_{n\to \infty}\int_0^{T}\int_{\mathbb{R}^{2d}}f_{\varepsilon_n}\mathrm{div}_{t,x,p}\phi dxdpdt\right|\\
&=\left|\lim_{n\to \infty}\int_0^{T}\int_{\mathbb{R}^{2d}}\nabla_{t,x,p}f_{\varepsilon_n}\cdot \phi dxdpdt\right|\\
&\leq \lim_{n\to \infty}\int_0^{T}\int_{\mathbb{R}^{2d}}|\nabla_{t,x,p}f_{\varepsilon_n}| dxdpdt ,
\end{align*}
where the last inequality follows from $||\phi||_{L^{\infty}((0,T)\times \mathbb{R}^{2d})}\leq 1.$
Since the right hand side of this equations is bounded by a constant independent of $\varepsilon_n$, we conclude that the total variation of $f$
\begin{align*}
\sup_{\phi}\left\{\int_0^{T}\int_{\mathbb{R}^{2d}}f\mathrm{div}_{t,x,p}\phi dxdpdt:\,\,\, \phi\in [C^1_c((0,T)\times \mathbb{R}^{2d})]^{2d+1}\,\,\, \text{and}\,\,\, ||\phi||_{L^{\infty}((0,T)\times \mathbb{R}^{2d})}\leq 1\right\}
\end{align*}
is finite. This fact and \cite[Proposition 3.6]{Amb} imply that $f\in BV((0,T)\times \mathbb{R}^{2d})$ for every $T>0$.

Lemma \ref{l:epsilon} shows that $||\nabla_p f_{\varepsilon_n}||_{L^2((0,T)\times \mathbb{R}^{2d})}$ is uniformly bounded  for every $T>0$.  Without loss generality, we may assume that the entire sequence $\nabla_p f_{\varepsilon_n}$ converges weakly to a vector function $F\in [L^2((0,T)\times \mathbb{R}^{2d})]^d$ . Let $\psi $ be a smooth compactly supported function in  $(0,T)\times \mathbb{R}^{2d}).$ 
Since  $ f_{\varepsilon_n}\to f$ in $L^1_{\mathrm{loc}}((0,\infty)\times \mathbb{R}^{2d})$ and $\nabla_p f_{\varepsilon_n} \rightharpoonup F$ as $n\to \infty$,   we have
\begin{align*}
\int_0^T\int_{\mathbb{R}^{2d}}f \nabla_p \psi dxdt&=\lim_{n\to \infty} \int_0^T\int_{\mathbb{R}^{2d}}f_{\varepsilon_n}\nabla_p \psi dxdt\\
&=-\lim_{n\to \infty} \int_0^T\int_{\mathbb{R}^{2d}}\nabla_p f_{\varepsilon_n} \psi dxdt\\
&=-\int_0^T\int_{\mathbb{R}^{2d}}F  \psi dxdt.
\end{align*}
Thus, by the definition of weak derivatives
$\nabla_p f=F\in [L^2((0,T)\times \mathbb{R}^{2d})]^d. $

We write  \eqref{epdef}  for $f_{\varepsilon_n}$
\begin{align*}
\int_0^{\infty}\int_{\mathbb{R}^{2d}}\mathrm{sign}(f_{\varepsilon_n}-a)&\big\{(f_{\varepsilon_n}-a)\partial_t \varphi +(f_{\varepsilon_n}-f^2_{\varepsilon_n}-(a-a^2))(p\cdot \nabla_x \varphi-p\cdot \nabla_p \varphi)\nonumber \\
&-\nabla_p f_{\varepsilon_n}\cdot\nabla_p \varphi+\varepsilon (f_{\varepsilon_n}-a) \Delta_x \varphi+d(a-a^2)\varphi \big\}dxdpdt\geq 0,
\end{align*}
where $\varphi\in C_c^{\infty}((0, \infty)\times \mathbb{R}^{2d}) $ with $\varphi\geq 0.$ As $\varphi$ has a compact support, we can pass to the limit  and get  \begin{align*}
\int_0^{\infty}\int_{\mathbb{R}^{2d}}\mathrm{sign}(f-a)&\big\{(f-a)\partial_t \varphi +(f-f^2-(a-a^2))(p\cdot \nabla_x \varphi-p\cdot \nabla_p \varphi)\nonumber \\
&-\nabla_p f\cdot\nabla_p \varphi+d(a-a^2)\varphi \big\}dxdpdt\geq 0.
\end{align*}
 It remains to show that $ \lim_{t \searrow 0} \int_{\mathbb{R}^{2d}}|f(t,x,p)-f_0(x,p)|dxdp=0.$ The function $f_{\varepsilon_n}-f_0$ satisfies 
 \begin{align*}
 \partial_t (f_{\varepsilon_n}-f_0)=&-p\cdot \nabla_x (f_{\varepsilon_n}-f^2_{\varepsilon_n}-(f_0-f_0^2))+\mathrm{div}_{p}(\nabla_p f_{\varepsilon_n}+p(f_{\varepsilon_n}-f^2_{\varepsilon_n}-(f_0-f_0^2)))\\
 &+\varepsilon \Delta_x (f_{\varepsilon_n}-f_0)
 +L[f_0]+\varepsilon \Delta_x f_0.
 \end{align*}
 By repeating the arguments in the proof of Lemma \ref{lem:L1}, one can prove
 $$||f_{\varepsilon_n}(t)-f_0||_{L^1(\mathbb{R}^{2d})}\leq  t||L[f_0]||_{L^1(\mathbb{R}^{2d})}+t\varepsilon ||\Delta_x f_0||_{L^1(\mathbb{R}^{2d})},\,\,\,\,\, \forall\,t>0.$$ Since $f_{\varepsilon_n}-f_0$ converges to $f-f_0$ pointwise for almost everywhere, Fatou's lemma provides 
 $$||f(t)-f_0||_{L^1(\mathbb{R}^{2d})}\leq  t||L[f_0]||_{L^1(\mathbb{R}^{2d})}$$
 for almost every $t\in (0,\infty).$ Hence, there  is a set $\mathcal{E}  $  of zero measure on $(0,1)$ such that 
$$ \lim_{t \to 0,\, t\in (0,1) \setminus \mathcal{E}}||f(t)-f_0||_{L^1(\mathbb{R}^{2d})}=0.$$
The above properties of $f$  shows that it is an entropy solution. Its uniqueness follows from Theorem \ref{th:unique}.

 For all $\gamma\in [1,k],$ the $\gamma$-moment $\int_{\mathbb{R}^{2d}}|p|^{\gamma}f_{\varepsilon_n}(t,x,p)dxdp$ is bounded by a polynomial $P_{\lceil \gamma/2\rceil}(t)$ (independent of $\varepsilon_n$) of degree $\lceil \gamma/2\rceil $.  
This fact and  Fatou's lemma imply that $\int_{\mathbb{R}^{2d}}|p|^{\gamma}f(t,x,p)dxdp$ is also bounded by the same polynomial for almost every $t\in (0,\infty)$.

 Let $\zeta_{\sigma}$ be the function defined in \eqref{zeta}. The proof of Lemma \ref{l:mass} shows that 
  \begin{align*}
\int_{\mathbb{R}^{2d}}\zeta_{\sigma}f_{\varepsilon_n}(T)dpdx-\int_{\mathbb{R}^{2d}}\zeta_{\sigma}f_0dpdx&=\int_0^T\int_{\mathbb{R}^{2d}} (\Delta_p \zeta_{\sigma} f_{\varepsilon_n}-\nabla_p \zeta_{\sigma}\cdot p(f_{\varepsilon_n}- f_{\varepsilon_n}^2))dpdxdt\\
&+\int_0^T\int_{\mathbb{R}^{2d}} (\varepsilon_n \Delta_x \zeta_{\sigma} f_{\varepsilon_n}+\nabla_x \zeta_{\sigma}\cdot p(f_{\varepsilon_n}- f_{\varepsilon_n}^2))dpdxdt.
\end{align*}
 Since $ f_{\varepsilon_n}$ converges to  $f$ in $L^1_{\mathrm{loc}}((0,\infty)\times \mathbb{R}^{2d})$ as $n\to \infty,$ the right hand side of this equation converges to $$\int_0^T\int_{\mathbb{R}^{2d}} (\Delta_p \zeta_{\sigma} f-\nabla_p \zeta_{\sigma}\cdot p(f- f^2))dpdxdt
+\int_0^T\int_{\mathbb{R}^{2d}} \nabla_x \zeta_{\sigma}\cdot p(f- f^2)dpdxdt$$ as $n\to \infty$ for every $T>0.$ This implies that $\int_{\mathbb{R}^{2d}}\zeta_{\sigma}f_{\varepsilon_n}(T)dpdx$ has a limit for every $T>0$. It also converges to $\int_{\mathbb{R}^{2d}}\zeta_{\sigma}f(T)dpdx$ for almost every $T>0.$ 
 Hence, we have  obtained
  \begin{align*}
\int_{\mathbb{R}^{2d}}\zeta_{\sigma}f(T)dpdx-\int_{\mathbb{R}^{2d}}\zeta_{\sigma}f_0dpdx&=\int_0^T\int_{\mathbb{R}^{2d}} (\Delta_p \zeta_{\sigma} f-\nabla_p \zeta_{\sigma}\cdot p(f- f^2))dpdxdt\\
&+\int_0^T\int_{\mathbb{R}^{2d}} \nabla_x \zeta_{\sigma}\cdot p(f- f^2)dpdxdt
\end{align*}
for almost every $T>0.$
Using $0\leq f\leq 1$, $|p|f\in L^1([0,T]\times \mathbb{R}^{2d})$, and the properties of $\zeta_{\sigma}$ in \eqref{zeta},  we let $\sigma\to 0$ and get
$$\int_{\mathbb{R}^{2d}}f(T)dpdx=\int_{\mathbb{R}^{2d}}f_0dpdx$$
for almost every $T>0.$
\end{proof}

 \section{Conclusion and Outlook}

In this paper, we have studied   the spatially inhomogeneous Fermi-Dirac-Fokker-Planck equation \eqref{DEq}. In contrast to the existing results for spatially inhomogeneous equation \cite{NS,Chi}, our analysis does not require the initial datum to be sufficiently close to equilibrium or to satisfy a smallness condition. 
Our main results can be summarized as follows:
\begin{enumerate}
\item[(a)] \textbf{Entropy solutions.} We have introduced a notion of entropy solutions which are $BV$ functions and  satisfy \eqref{DEq} in the sense of distributions and in the sense of equality of locally finite Radon measures.  We have proven an $L^1$-contraction principle for entropy solutions, which in particular yields uniqueness for a given initial datum.
\item[(b)] \textbf{Global existence.}
We have proved the global existence and uniqueness of entropy solutions for  initial data satisfying suitable integrability, regularity, and moment assumptions. The existence theory is based on a vanishing-viscosity approach in which a Laplacian in the spatial variable is added to the equation. We derive estimates that are uniform with respect to the viscosity parameter and use compactness to pass to the limit. The limiting function satisfies the entropy formulation, and by the uniqueness theorem, it is the unique entropy solution. 
\item[(c)] \textbf{Properties of the solution.} We have proven that the  solution preserves the total mass and satisfies polynomial-in-time bounds for its  moments with respect to the $p$-variable. 
\end{enumerate}
The results show that, despite the degeneracy in the spatial variable and the presence of the nonlinear term $f(1-f)$ appearing in  both transport and Fokker-Planck operators, one can obtain a global well-posedness theory without relying on perturbative arguments. The vanishing viscosity approach provides a natural framework to construct global solutions under some regularity assumptions on the initial data.

Several questions remain open. In particular, it would be interesting to investigate whether the assumptions on the initial data can be weakened, whether the moment condition $k>2d$ used in the global construction is optimal, and whether the entropy solutions constructed here enjoy additional regularity, in particular hypoelliptic regularization. Another interesting question concerns the long-time behavior of the solution. It is expected that the solution converges to a  Fermi-Dirac distribution; see \cite{NS, APT, Chi}. Understanding the large-time behavior and obtaining explicit  rates on convergence towards the Fermi-Dirac distribution for general, non-small initial data is a natural direction for future work.

More broadly, quantum (or semiclassical) kinetic equations are mathematically more intricate than their classical counterparts because  quantum statistical effects introduce additional nonlinearities and lead to the Fermi-Dirac and Bose-Einstein equilibrium states.
We expect that the techniques developed in this paper may also be useful for the existence analysis  of other quantum kinetic equations.

 
 \bigskip

\textbf{Acknowledgement.}
The author was funded  by the Deutsche Forschungsgemeinschaft (DFG, German Research Foundation) under Germany's Excellence Strategy EXC 2044/2--390685587, Mathematics Münster: Dynamics--Geometry--Structure. 
 \bigskip
 
\textbf{Data Availability.} There is no data associated with the paper.
 \bigskip

\textbf{Conflict of interest.} The author has no conflict of interest to declare.
 \bigskip
 
 
{}

\end{document}